\documentclass{aimscleaned}

\usepackage{psfrag}

\usepackage{amsmath}
  \usepackage{paralist}
  \usepackage{graphics} 
  \usepackage{epsfig} 
\usepackage{graphicx}  \usepackage{epstopdf}
 \usepackage[colorlinks=true]{hyperref}
\hypersetup{urlcolor=blue, citecolor=red}

\def\currentvolume{X}

    \def\ppages{X--XX}
     \def\DOI{10.3934/xx.xx.xx.xx}

\graphicspath{{./figure/}{./}}

\newcommand{\R}{\mathbb R}
\newcommand{\Dx}{{\Delta x}}

\newcommand{\rhomax}{\rho_{\textup{max}}}
\newcommand{\DtsuDx}{\frac{\Delta t}{\Delta x}}

\newcommand{\alphauno}{u_\ell}
\newcommand{\alphadue}{v_\ell}
\renewcommand{\beta}{w_r}

\newcommand{\rhomeno}{\rho_\ell}
\newcommand{\rhopiu}{\rho_r}
\newcommand{\rholeft}{\rho_\ell}
\newcommand{\rhoright}{\rho_r}

\newcommand{\hL}{h_\ell}
\newcommand{\hR}{h_r}
\newcommand{\hc}{h_c}

\newcommand{\compl}{\sharp}

\newcommand{\ul}{u_\ell}
\newcommand{\vl}{v_\ell}
\newcommand{\zc}{z_c}
\renewcommand{\wr}{w_r}

\newcommand{\Frac}{\displaystyle\frac}

\newtheorem{theorem}{Theorem}[section]

\newtheorem{proposition}{Proposition}

\theoremstyle{definition}
\newtheorem{definition}[theorem]{Definition}
\newtheorem{remark}{Remark}
\newtheorem*{notation}{Notation}

\title[An easy-to-use algorithm for simulating traffic flow] 
      {An easy-to-use algorithm for simulating traffic flow on networks: theoretical study}

\author[M. Briani and E. Cristiani]{}

\subjclass{Primary: 65M08; Secondary: 90B20.}
\keywords{LWR model, Godunov scheme, multi-population model, multi-path model, multi-commodity model, multi-class model, source-destination model, buffer model, traffic flow, networks.}

 \email{m.briani@iac.cnr.it}
 \email{e.cristiani@iac.cnr.it}

\thanks{The research leading to these results has received funding from the European Union FP7 under grant No.\ 257462 HYCON2 Network of Excellence. The research was also supported by the collaboration with the private company ZEROPIU (Italy) and by the Google Research Award ``Multipopulation Models for Vehicular Traffic and Pedestrians'', 2012-2013.}

\begin{document}
\maketitle

\centerline{\scshape Maya Briani}
\medskip
{\footnotesize
 \centerline{IAC-CNR}
   \centerline{Via dei Taurini, 19}
   \centerline{Rome, Italy}
} 

\medskip

\centerline{\scshape Emiliano Cristiani}
\medskip
{\footnotesize
 \centerline{IAC-CNR}
   \centerline{Via dei Taurini, 19}
   \centerline{Rome, Italy}
}

\bigskip

 \centerline{(Communicated by the associate editor name)}

\begin{abstract}
In this paper we study a model for traffic flow on networks based on a hyperbolic system of conservation laws with discontinuous flux. Each equation describes the density evolution of vehicles having a common path along the network. In this formulation the junctions disappear since each path is considered as a single uninterrupted road.

We consider a Godunov-based approximation scheme for the system which is very easy to implement. Besides basic properties like the conservation of cars and positive bounded solutions, the scheme exhibits other nice properties, being able to select automatically a solution at junctions without requiring external procedures (e.g., maximization of the flux via a linear programming method). Moreover, the scheme can be interpreted as a discretization of the traffic models with buffer, although no buffer is introduced here.

Finally, we show how the scheme can be recast in the framework of the classical theory of traffic flow on networks, where a conservation law has to be solved on each arc of the network. This is achieved by solving the Riemann problem for a modified equation, and showing that its solution corresponds to the one computed by the numerical scheme.
\end{abstract}

\section{Introduction}\label{sec:introduction}
Starting from the introduction of the LWR model \cite{LW, R}, a huge literature about macroscopic fluid-dynamic models for traffic flow was developed. More recently, models, theory and numerical approximations for traffic flow on networks became a hot topic \cite{coclite2005SIMA, daganzo1995TRBa, piccolibook, holden1995SIMA}. The interest in forecasting traffic flow on large networks became even stronger in the very last years, due to the increasing number of GPS devices (smartphones, satellite navigators, black boxes) which provide real-time traffic data. Private companies like GOOGLE, WAZE MOBILE, NOKIA \cite{herrera2010TRB}, INRIX, OCTOTELEMATICS \cite{cristiani2010CAIM}, ZEROPIU, YANDEX, started collecting data and, in some cases, broadcasting traffic forecast.

The LWR model describes the evolution of the average density of vehicles on a road by means of a conservation law of the form
\begin{equation}\label{CL}
\frac{\partial}{\partial t}\rho+\frac{\partial}{\partial x}f(\rho)=0, \qquad x\in (a,b), \quad t>0
\end{equation}
where $\rho(x,t)$ is the vehicle density at point $x$ and time $t$ and $f=f(\rho(x,t))$ is a given flux function. Extension of such model to road networks (i.e., graphs) is not straightforward since the dynamics at junctions is not uniquely determined. Indeed, imposing the conservation of vehicles at junctions and complying with drivers' preferences with regards the desired directions still leaves infinite admissible solutions for the density.

\medskip

\noindent \textit{Goal.} Preliminarily, we show the relationship between a first-order version of the model proposed by Hilliges and Weidlich \cite{hilliges1995TRB} and the source-destination model proposed by Garavello and Piccoli \cite{garavello2005CMS}. The main feature of these models is that they are able to track different populations of drivers, characterized by their path along the network, thus keeping a global view of the vehicular flow. 
The Hilliges and Weidlich's model, hereafter called \textit{multi-path}, is particularly interesting since the junctions apparently disappear. This is made possible by the fact that each path is considered as a single uninterrupted road, and the junctions are hidden in the coupling of the equations which describe the densities of the single populations.

The main goal of this paper is investigating the properties of a new Godunov-based numerical scheme for the multi-path model, recently proposed in the ``twin'' paper \cite{bretti2013DCDS-S}, showing unexpected similarities to traffic models with \textit{buffer} \cite{bressan2014preprint,garavello2012DCDS,garavello2013bookchapt,herty2009NHM}. Considering that the proposed numerical scheme prescribes a unique solution for the density at junction, the question arises which solution is automatically selected among the admissible ones. We also give possible answers to this question.

\medskip

\noindent \textit{Relevant literature.} First of all, let us stress that the multi-path model differs from the so-called \emph{multi-population} or \emph{multi-class} models, see, e.g., \cite{benzoni2003EJAM, wong2002TRA}. In those cases, the models consist of one equation for a single road (extension to networks is also possible) with different velocity functions $v_i$, one for each class of vehicles. Typically, the populations have different maximal velocities, in order to take into account different types of vehicles or drivers' behaviors.

The multi-path model shares instead the same purpose of the source-destination model \cite{garavello2005CMS} (see also \cite{herty2008CMS}, \cite[Sect.\ 7]{lebacque1996proc} and the recent paper \cite{bressan2014preprint}). In that case, vehicles are divided in different populations on the basis on their source-destination pair. The model consists of a system of nonlinear and semilinear PDEs, one for the total density $\rho$ and the others for the percentages $\pi$'s of vehicles belonging to the different populations.

Let us also mention the papers \cite{bressan2014preprint,garavello2012DCDS,garavello2013bookchapt,herty2009NHM} which deal with \textit{buffer} models. In that case the junction is seen as a 0-dimensional space and its load is described by the number of cars lying in it at any time. The load varies accordingly to the difference between the inflow and the outflow at the buffer and evolves according to a dedicated ODE.

The multi-path model is based on a system of nonlinear conservation law with discontinuous flux. Several papers investigate from the theoretical point of view (systems of) scalar conservation laws. The interested reader can find in the book \cite{levequebook} an introduction to the field, in \cite{adreianov2011ARMA,burger2008JEM} some references for the case of discontinuous flux, and in the book \cite{bressanbook} the analysis of the systems of conservation laws. Systems of scalar conservation laws with discontinuous flux are instead less studied. An attempt related to traffic flow can be found in \cite{mercier2009JMAA}, where a model very similar to the one considered here is investigated.
From the numerical point of view, a good basic reference is again the book \cite{levequebook}. We also point out the paper \cite{herty2007ANM}, where a numerical method for (systems of) scalar conservation laws with discontinuous flux is proposed, and the paper \cite{towers2000SINUM}, where the convergence of a Godunov-based scheme for scalar conservation laws with discontinuous flux is investigated.

\medskip

\noindent \textit{Paper organization.} In Section 2 we describe the multi-path model, pointing out the analogies with the model introduced in \cite{garavello2005CMS}. We also introduce a junction-oriented version of the model particularly suited for large networks and real applications.

In Section 3 we introduce the Godunov-based numerical scheme and we show that it is conservative, and its solution is positive and bounded by the maximal admissible density on the roads.

In Section 4 we show how the numerical scheme basically acts as a discretization of the models with buffer, although no buffer is explicitly introduced here. 

In Section 5 we focus on the case of a merge (a single junction with 2 incoming roads and 1 outgoing road). Since it is well known that in this case the LWR model on networks admits infinite solutions at junction (see, e.g., \cite{piccolibook}), we try to understand which solution is automatically selected by the scheme among the admissible ones. To this end, we first compute the asymptotic solution of the numerical scheme for some constant initial/boundary conditions and then we propose a new set of equations, compatible with the classical theory of traffic flow on networks (where a single equation is solved on any arc of the graph), such that the solution of the Riemann problem associated to those equations coincides with the solution of the numerical scheme. Finally we present some numerical results in order to confirm the theoretical findings.

In Section 6 we sketch some conclusions and in the Appendices we report the proofs of the main theorems and some basic notions about the (half) Riemann problem.

\section{The multi-path model}\label{sec:model}
In this section we introduce a first-order version of the model proposed in \cite[Sect.4]{hilliges1995TRB}, which will be studied in the next sections. 

Let us consider a network, i.e.\ a directed graph with $N_R$ arcs (roads) and $N_J$ nodes (junctions). \textit{Vehicles moving on the network are divided on the basis of their path}. Let us assume that the number of possible paths on the graph is $N_P$ and denote those paths by $P^1,\ldots,P^p,\ldots,P^{N_P}$, see Fig.\ \ref{fig:network}. 
\begin{figure}[h!]
\begin{center}
\begin{psfrags}
\psfrag{P1}{$P^1$} \psfrag{P2}{$P^2$}
\includegraphics[width=0.47\textwidth]{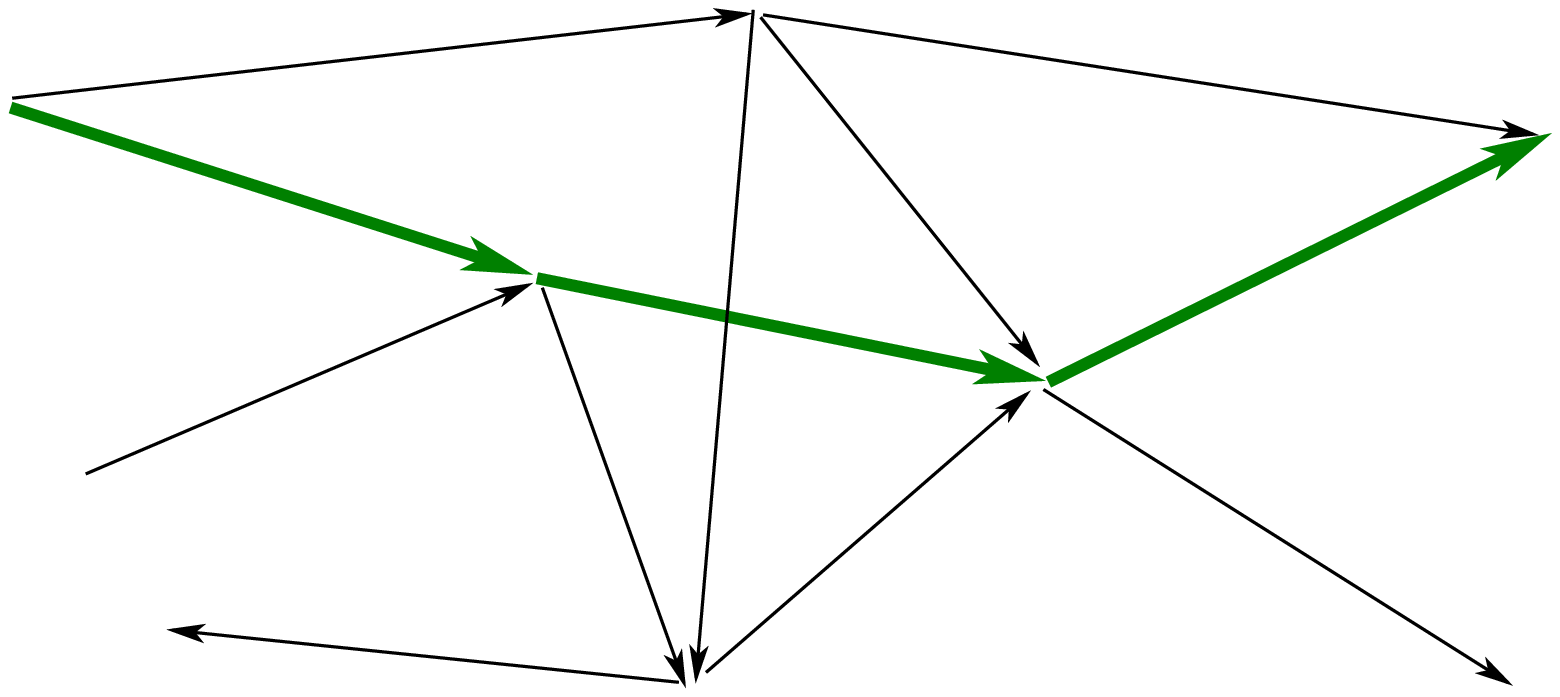}\qquad
\includegraphics[width=0.47\textwidth]{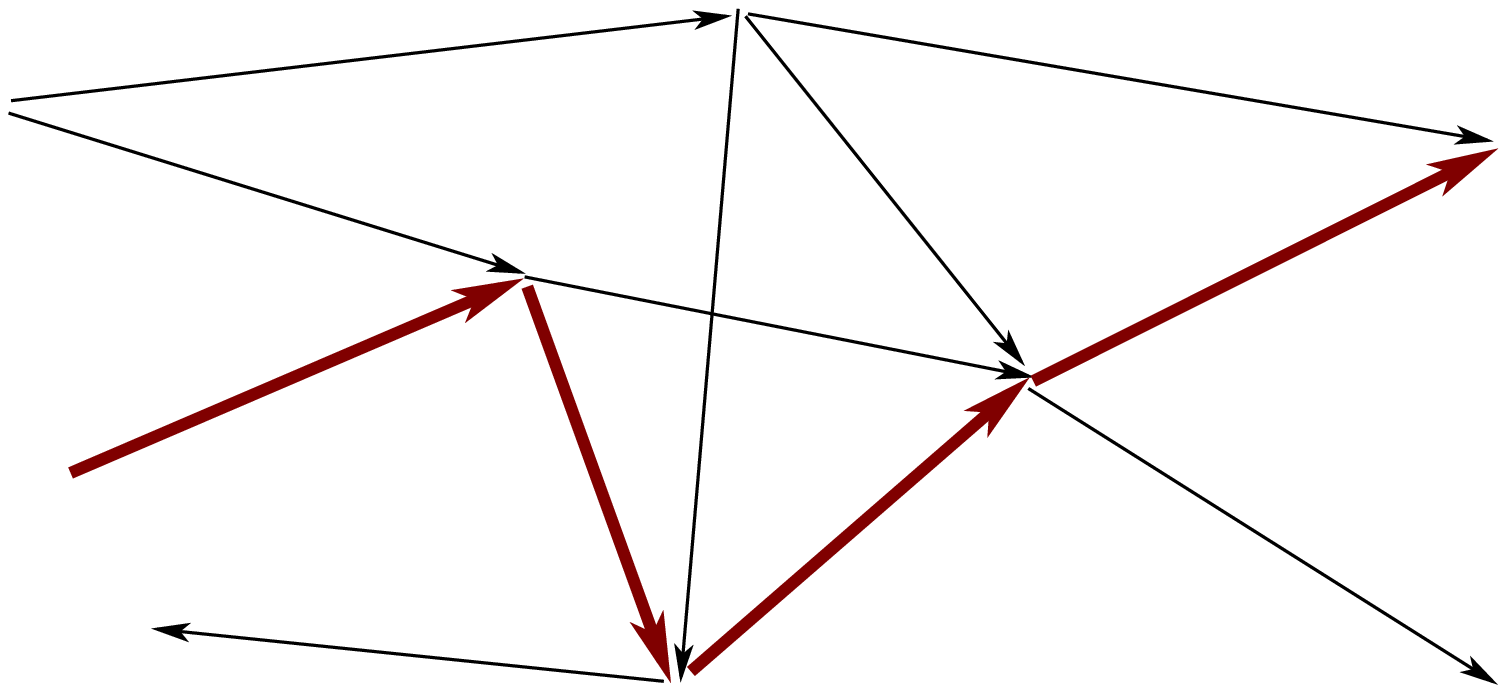}
\end{psfrags}
\end{center}
\caption{A generic network. Two possible paths are highlighted. Note that the two paths share an arc of the network.}
\label{fig:network}
\end{figure}
We stress that paths can share some arcs of the network. A point $x^{(p)}$ of the network is characterized by both the path $p$ it belongs to and the distance $x$ from the origin of that path. We denote by $\mu^p(x^{(q)},t)$ the density of the vehicles following the $p$-th path at point $x^{(q)}$ at time $t>0$, and we assume that $\mu^p(x^{(q)},t) \in [0,\rhomax]$ for some maximal density $\rhomax$. Note that we have, by definition, $\mu^p(x^{(q)},t)=0$ if $x^{(q)}\notin P^p$. 
We also define
\begin{equation}\label{def:omega}
\omega(x^{(p)},t):=\sum_{q=1}^{N_P}\mu^q(x^{(p)},t),
\end{equation}
i.e.\ $\omega(x^{(p)},t)$ is the sum of all densities living at $x^{(p)}$ at time $t$.
The function $\omega$ is expected to be discontinuous, especially at junctions, and contains all the information about the topology of the network.
Note that, for any point $x^{(q)}$, the densities $\mu^p(x^{(q)},t)$, $p=1,\ldots,N_P$, are admissible if their sum $\omega(x^{(q)},t)\leq\rhomax$.
Let us denote by $v(\omega)$ the velocity of vehicles (given as a function of the density) and by $f(\omega)=\omega v(\omega)$ the flux of vehicles. 
The LWR-based multi-path model is constituted by the following system of $N_P$ conservation laws with space-dependent and discontinuous flux 
\begin{equation}\label{eq:system_v}
\frac{\partial}{\partial t}\mu^p(x^{(p)},t)+
\frac{\partial}{\partial x^{(p)}}\left(\mu^p(x^{(p)},t)\ v\big(\omega(x^{(p)},t)\big)\right)=0, \quad x^{(p)}\in P^p,\ t>0,
\end{equation}
or, equivalently,
\begin{equation}\label{eq:system_f}
\frac{\partial}{\partial t}\mu^p(x^{(p)},t)+
\frac{\partial}{\partial x^{(p)}}\left(\frac{\mu^p(x^{(p)},t)}{\omega(x^{(p)},t)}\ f\big(\omega(x^{(p)},t)\big)\right)=0, \quad x^{(p)}\in P^p,\ t>0,
\end{equation}
for $p=1,\ldots,N_P$.
If $\omega=0$ we have, \textit{a fortiori}, $\mu^p=0$, then it is convenient to set $\frac{\mu^p}{\omega}=0$ in (\ref{eq:system_f}) to avoid singularities.
In the following we assume that the flux $f\in C^0([0,\rhomax])\cap C^1((0,\rhomax))$ and
\begin{equation}\label{proprieta_f}
f(0)=f(\rhomax)=0,\qquad f \textrm{ is strictly concave}, \qquad \sigma:=\arg\max_{\omega\in (0,\rhomax)}f(\omega). 
\end{equation}
Equations of the system (\ref{eq:system_v}) (or (\ref{eq:system_f})) are coupled by means of the velocity $v$, which depends on the total density $\omega$. On the other hand, not all the equations of the system are coupled with each other because paths do not have necessarily arcs in common. Note that in this formulation the junctions do not appear explicitly, since every equation is solved in the \textit{uninterrupted} one-dimensional domain $P^p$. Junctions are actually hidden in the function $\omega$, since the indices of non-zero densities $\mu^q$'s in its definition \eqref{def:omega} change abruptly at junctions. 

\medskip

As a preliminary remark, it is important to show the strict relationship between the multi-path model and the source-destination model \cite{garavello2005CMS}. 
Let us focus on a single road $(a,b)$ and denote by $\rho$ the total density $\omega$. Let us also define, for any $\xi\in(a,b)$ and $t>0$, $\pi^p(\xi,t):=\frac{\mu^p(\xi,t)}{\rho(\xi,t)}$ as the fraction of vehicles following the path $p$. Then, from \eqref{eq:system_v}, we formally have
\begin{eqnarray*}
\frac{\partial}{\partial t}\mu^p+\frac{\partial}{\partial x}(\mu^p v(\rho))=0 \\
\Leftrightarrow
\frac{\partial}{\partial t}(\pi^p \rho)+\frac{\partial}{\partial x}(\pi^p \rho v(\rho))=0 \\
\Leftrightarrow
\frac{\partial}{\partial t}(\pi^p) \rho+
\pi^p\frac{\partial}{\partial t}\rho+
\frac{\partial}{\partial x}(\pi^p) \rho v(\rho)+
\pi^p\frac{\partial}{\partial x}(\rho v(\rho))=0\\
\Leftrightarrow
\pi^p\left(\frac{\partial}{\partial t}\rho+\frac{\partial}{\partial x}(\rho v(\rho))\right)
+
\rho\left(\frac{\partial}{\partial t}\pi^p+\frac{\partial}{\partial x}(\pi^p)v(\rho)\right)=0
\end{eqnarray*}
which holds true if
\begin{equation}\label{eq:sys_source-dest}
\left\{
\begin{array}{l}
\frac{\partial}{\partial t}\rho+\frac{\partial}{\partial x}(\rho v(\rho))=0,\\ [2mm]
\frac{\partial}{\partial t}\pi^p+v(\rho)\frac{\partial}{\partial x}\pi^p=0,
\end{array}
\right.
\end{equation}
i.e.\ the source-destination model.

A similar system of equations appears in the buffer model proposed in \cite{bressan2014preprint}, where the authors study the Cauchy problem defined by the system \eqref{eq:sys_source-dest} supplemented by ODEs describing the state of the buffer at junctions. We will see in Section \ref{sec:buffer} that our numerical approach based on the approximation of \eqref{eq:system_f} basically corresponds to a discretization of the models with buffer.

\section{The numerical scheme and its basic properties}\label{sec:scheme_and_properties}
In this section we present the numerical approximation we use to discretize the system (\ref{def:omega}),(\ref{eq:system_f}). The same approximation is also considered in the ``twin'' paper \cite{bretti2013DCDS-S}, together with some numerical results. 

For any path $p$, we define a numerical grid in $P^p\times[0,+\infty)$ with space step $\Delta x>0$ and time step $\Delta t>0$.

\subsection{The scheme and the algorithm}\label{sec:scheme}
Let us denote by $x_k^{(q)}:=k^{(q)}\Delta x$, $k^{(q)}\in\mathbb Z$, the center of the $k^{(q)}$-th space cell along the path $P^q$, and by $t^n:=n\Delta t$, $n\in\mathbb N$, the center of the $n$-th time cell. Let us also denote by $\mu^{n,p}_{k^{(q)}}$ the approximate density $\mu^p(x_k^{(q)},t^n)$. Then, analogously to (\ref{def:omega}), we define
\begin{equation}\label{def:omega_approx}
\omega^{n}_{k^{(p)}}:=\sum_{q=1}^{N_P}\mu^{n,q}_{k^{(p)}}.
\end{equation}
From now on, to avoid cumbersome notations, we write $k^p$ instead of $k^{(p)}$.

Equation (\ref{eq:system_f}) is discretized by means of the following Godunov-type scheme, which reads, at any internal cell $k^p$, as
\begin{equation}\label{schema}
\mu_{k^p}^{n+1,p} = \mu_{k^p}^{n,p}-
\DtsuDx\left(\frac{\mu_{k^p}^{n,p}}{\omega_{k^p}^{n}} \ G_f(\omega_{k^p}^{n},\omega_{k^p+1}^{n})- 
\frac{\mu_{k^p-1}^{n,p}}{\omega_{k^p-1}^{n}}\ G_f(\omega_{k^p-1}^{n},\omega_{k^p}^{n})
\right)
\end{equation}
for $n\geq 0$ and $p=1,\ldots,N_P$, where $G_f$ is the classical $f$-based Godunov numerical flux defined, as usual \cite{levequebook}, as
\begin{equation}\label{GodunovFlux}
G_f(\rhomeno,\rhopiu):=
\left\{
\begin{array}{ll}
\min\{f(\rhomeno),f(\rhopiu)\} & \textrm{if } \rhomeno\leq \rhopiu \\
f(\rhomeno) & \textrm{if } \rhomeno>\rhopiu \textrm{ and } \rhomeno<\sigma \\
f(\sigma) & \textrm{if } \rhomeno>\rhopiu \textrm{ and } \rhomeno \geq \sigma \geq \rhopiu \\
f(\rhopiu) & \textrm{if } \rhomeno>\rhopiu \textrm{ and } \rhopiu>\sigma.
\end{array}
\right.
\end{equation}
In the following we drop the subscript $f$ from $G$ whenever the underlying flux is deducible without ambiguity.
Note the intrinsic asymmetry of this scheme: Coefficients in front of the fluxes involve only the cells $k^p$ and $k^p-1$, and not $k^p+1$.
 
The algorithm can be summarized as follows: 
\begin{enumerate}
\item Set the initial conditions $\mu^{0,p}_{k^p}$ for any $p$ at any internal cell $k^p$.
\item Set the boundary conditions at the beginning and the end of the $N_P$ paths at any time step $n$. 
\item Set $n=0$.
\item Compute $\omega^n_{k^p}$ at any internal cell $k^p$ by means of \eqref{def:omega_approx}.
\item Compute $\mu^{n+1,p}_{k^p}$ for any $p$ at any internal cell $k^p$ by means of \eqref{schema}.
\item If the final time is not reached go to step 4 with $n\leftarrow n+1$.
\end{enumerate}

We stress again that no special management of the junctions is needed. The simplicity of the scheme is the main strength of this approach, making it possible to simulate traffic flow on networks in minutes.

\begin{remark} 
If the network under consideration is small, the number of possible paths is reasonably small. Then, the number of equations in the system (\ref{eq:system_f}) fits a manageable size. Conversely, if the network is large and allows for a large number of paths, the computation of $\omega$ in (\ref{def:omega_approx}) becomes a hard task. In this case, one can adopt a hybrid point of view, creating a model which merges the features of the multi-path model with those of the classical models, where a PDE has to be solved on each arc of the network \cite{piccolibook}. In the internal cells of the arcs a single equation for the \emph{total} density is solved. 
Then, at the cell before each junction, vehicles' density is split on the basis of the desired direction of drivers (by means of some given distribution coefficients), and the scheme \eqref{schema} is applied. At the cell after the junction, sub-densities $\mu^p$'s are summed again. The price to pay is that the global behavior of drivers along the whole network is lost (cfr.\ also \cite{daganzo1994TRB, garavello2005CMS, tampere2011TRB} on this point). 
\end{remark}

\begin{remark} 
Although the scheme \eqref{schema} is clearly derived by the system \eqref{eq:system_f}, we do not state that the latter is consistent with the former as $\Delta x,\ \Delta t \to 0$. The derivation of the limit equation for the scheme \eqref{schema} is out of the scope of the paper. Rather, we prefer deriving a new set of equations which have the same solution of the numerical scheme, see Section \ref{sec:approfondimento2in1}.
\end{remark}

\subsection{Conservation of mass and admissibility of the solution}
It is immediate to prove that the scheme \eqref{schema} is \textit{conservative}, i.e.
$$
\sum_{k^p} \mu^{n+1,p}_{k^p}=\sum_{k^p} \mu^{n,p}_{k^p},\qquad \text{for any }n \text{ and } p
$$
and, \textit{a fortiori}, the sum $\omega$ is conserved in time (on both each road and the whole network).

\medskip

Another crucial property which the scheme must satisfy is that the single densities $\mu^p$'s and their sum $\omega$ never exceed $\rhomax$. The properties of the Godunov scheme guarantee that on each path the densities $\mu^p$'s are bounded by $\rhomax$, but this is no longer true for the \emph{sum} of the densities $\mu^p$, especially if two or more paths merge together.

To simplify the discussion, we first prove that the solution is bounded by $\rhomax$ in the case of a simple merge and a diverge. Generalizations will be discussed later on.

\subsubsection{Merge}\label{sec:merge}
Let us consider a network with three roads and one junction, with two incoming roads and one outgoing road. On this network two paths $P^1$ and $P^2$ are defined, see Fig.\ \ref{fig:network_2in_1out}. 
\begin{figure}[h!]
\begin{center}
\begin{psfrags}
\psfrag{P1}{$P^1$} \psfrag{P2}{$P^2$}
\psfrag{I1}{$[a_1,b_1]$} \psfrag{I2}{$[a_2,b_2]$} \psfrag{I3}{$[a_3,b_3]$}
\psfrag{J}{\tiny $J$} \psfrag{Jp}{\tiny $J\!\!+\!\!1$} \psfrag{Jm}{\tiny $J$-1}
\includegraphics[width=0.45\textwidth]{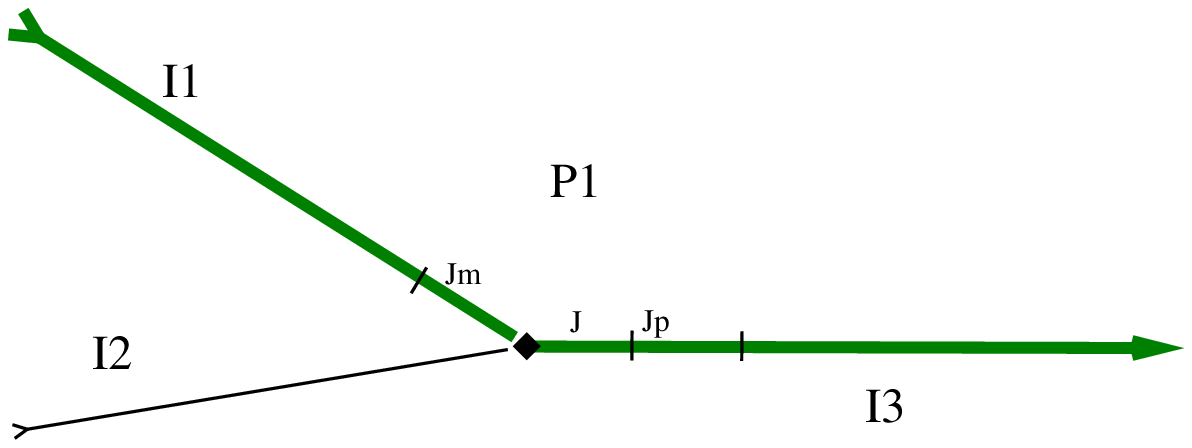}\quad
\includegraphics[width=0.45\textwidth]{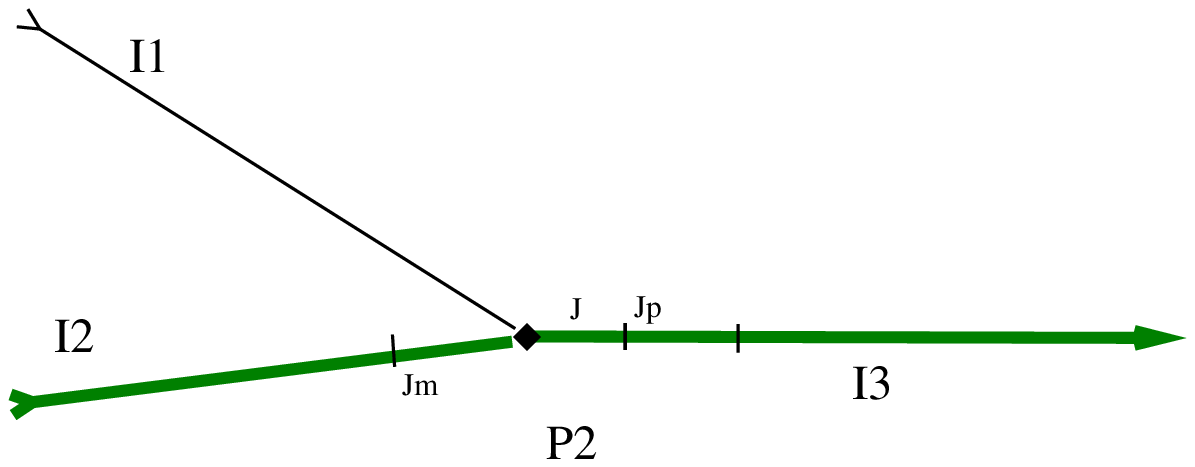}
\end{psfrags}
\end{center}
\caption{A network with 3 arcs and 1 junction, representing a merge. Path $P^1$ (left) and path $P^2$ (right).}
\label{fig:network_2in_1out}
\end{figure}
We denote by $J$ the grid cell \emph{just after} the junction, see Fig.\ \ref{fig:network_2in_1out}. Note that the cells before the junction, namely $J-1$, $J-2$, etc., can refer to one path or the other one, depending on the context.
We have
\begin{equation}\label{def_omega_2in1}
\omega_{k^1}^{n}=\left\{\begin{array}{ll}\mu_{k^1}^{n,1} & k^1<J, \\ \mu_{k^1}^{n,1}+\mu_{k^1}^{n,2} & k^1\geq J,\end{array}\right.
\qquad\qquad
\omega_{k^2}^{n}=\left\{\begin{array}{ll}\mu_{k^2}^{n,2} & k^2<J, \\ \mu_{k^2}^{n,1}+\mu_{k^2}^{n,2} & k^2\geq J,\end{array}\right.
\end{equation}
and the scheme (\ref{schema}) becomes
\begin{equation}\label{schema_2in1out}
\left\{\begin{array}{l}
\mu_{k^1}^{n+1,1} = \mu_{k^1}^{n,1}-\DtsuDx\left(\frac{\mu_{k^1}^{n,1}}{\omega_{k^1}^{n}}G(\omega_{k^1}^{n},\omega_{k^1+1}^{n})- \frac{\mu_{k^1-1}^{n,1}}{\omega_{k^1-1}^{n}}G(\omega_{k^1-1}^{n},\omega_{k^1}^{n})\right),\\
\\
\mu_{k^2}^{n+1,2} = \mu_{k^2}^{n,2}-\DtsuDx\left(\frac{\mu_{k^2}^{n,2}}{\omega_{k^2}^{n}}G(\omega_{k^2}^{n},\omega_{k^2+1}^{n})- \frac{\mu_{k^2-1}^{n,2}}{\omega_{k^2-1}^{n}}G(\omega_{k^2-1}^{n},\omega_{k^2}^{n})\right).
\end{array}\right.
\end{equation}

Let us now focus on the cell $J$, which is the only one in which the total density could exceed $\rhomax$ (standard properties of the Godunov scheme apply elsewhere).
\begin{theorem}\label{teo:rhoammissibile2in1}
Let the initial densities around the junction be admissible, namely
\begin{equation*}
\mu_{J-1}^{0,1}\leq\rhomax,\quad 
\mu_{J-1}^{0,2}\leq\rhomax,\quad 
(\mu_{J}^{0,1}+\mu_{J}^{0,2})\leq\rhomax,\quad
(\mu_{J+1}^{0,1}+\mu_{J+1}^{0,2})\leq\rhomax.
\end{equation*}
If the following CFL-like condition holds
\begin{equation}\label{CFLcon2}
2\DtsuDx \sup_{\rho\in(0,\rhomax)}|f'(\rho)|\leq 1,
\end{equation}
then
$$
(\mu_{J}^{n,1}+\mu_{J}^{n,2})\leq\rhomax \qquad \forall n.
$$
\end{theorem}

\begin{proof}
Let us first prove that \eqref{CFLcon2} implies
\begin{equation}\label{condizioneMaya}
\DtsuDx \leq \inf_{\rho\in[\sigma,\! \ \rhomax)}\frac{\rhomax-\rho}{2f(\rho)}.
\end{equation}
Let us define $M:=\sup_{\rho\in(0,\rhomax)}|f'(\rho)|$. By \eqref{CFLcon2} we have
\begin{equation}\label{lambda<=1su2M}
\DtsuDx\leq\frac{1}{2M}.
\end{equation}
Noting that $f(\rhomax)=0$, and using the Lagrange theorem and \eqref{lambda<=1su2M}, we have
$$
\inf_{\rho\in[\sigma,\! \ \rhomax)}\frac{\rhomax-\rho}{2f(\rho)}=
\inf_{\rho\in[\sigma,\! \ \rhomax)}\frac{|\rhomax-\rho|}{2|f(\rho)-f(\rhomax)|}\geq
\inf_{\rho\in[\sigma,\! \ \rhomax)}\frac{1}{2M}=\frac{1}{2M}\geq
\DtsuDx.
$$
To simplify the notations, let us introduce the auxiliary variable
$z^n:=\omega^n_J=\mu^{n,1}_{J}+\mu^{n,2}_{J}$.
The worst case happens when $\mu_{J-1}^{n,1}=\mu_{J-1}^{n,2}=\sigma$ (incoming roads try to transfer the maximal flux to cell $J$) and $\omega^n_{J+1}=\mu_{J+1}^{n,1}+\mu_{J+1}^{n,2}=\rhomax$ (no flux from cell $J$ to cell $J+1$). The equation for $z$ is
\begin{equation}\label{eq:merge_zn+1}
\begin{split}
z^{n+1}=z^n-\DtsuDx\Big(G(z^n,\omega^n_{J+1}) - G(\mu^{n,1}_{J-1},z^n) - G(\mu^{n,2}_{J-1},z^n) \Big)= \\ 
z^n-\DtsuDx\Big(G(z^n,\rhomax) - G(\sigma,z^n) - G(\sigma,z^n) \Big)= 
z^n+2\DtsuDx G(\sigma,z^n).
\end{split}
\end{equation}
We proceed by induction: Assume that $z^n\leq \rhomax$ and prove that $z^{n+1}\leq \rhomax$.
We have
$$
G(\sigma,z^n)=
\left\{
\begin{array}{ll}
f(\sigma) & \textrm{ if } z^n\leq\sigma, \\
f(z^n)    & \textrm{ if } z^n>\sigma.
\end{array}
\right.
$$
\begin{itemize}
\item CASE 1: $z^n\leq\sigma$\\
We have
$$
z^{n+1}=
z^n+2\DtsuDx f(\sigma)\leq
\sigma+2\DtsuDx f(\sigma).
$$
The conclusion follows easily by \eqref{condizioneMaya}, in particular by the fact that
$$
\DtsuDx \leq \frac{\rhomax-\sigma}{2 f(\sigma)}.
$$
\item CASE 2: $z^n>\sigma$
\begin{itemize}
\item CASE 2.1: $z^n=\rhomax$\\
We have $f(z^n)=f(\rhomax)=0$ and then $z^{n+1}=z^n=\rhomax$.
\item CASE 2.2: $z^n<\rhomax$\\
We have
$$
z^{n+1}=
z^n+2\DtsuDx f(z^n).
$$
The conclusion follows easily by \eqref{condizioneMaya}, in particular by the fact that
$$
\DtsuDx \leq \frac{\rhomax-z^n}{2f(z^n)}\quad\text{for any } z^n \in(\sigma,\! \ \rhomax).
$$
\end{itemize}
\end{itemize}
\end{proof}

\subsubsection{Diverge} 
Let us consider a network with three roads and one junction, with one incoming road and two outgoing roads. On this network two paths $P^1$ and $P^2$ are defined, see Fig.\ \ref{fig:esempio_network_1in2out}.
\begin{figure}[h!]
\begin{center}
\begin{psfrags}
\psfrag{P1}{$P^1$} \psfrag{P2}{$P^2$}
\psfrag{J}{\tiny $J$} \psfrag{Jp}{\tiny $J\!\!+\!\!1$} \psfrag{Jm}{\tiny $J$-1}
\psfrag{I1}{$[a_1,b_1]$} \psfrag{I2}{$[a_2,b_2]$} \psfrag{I3}{$[a_3,b_3]$}
\includegraphics[width=0.45\textwidth]{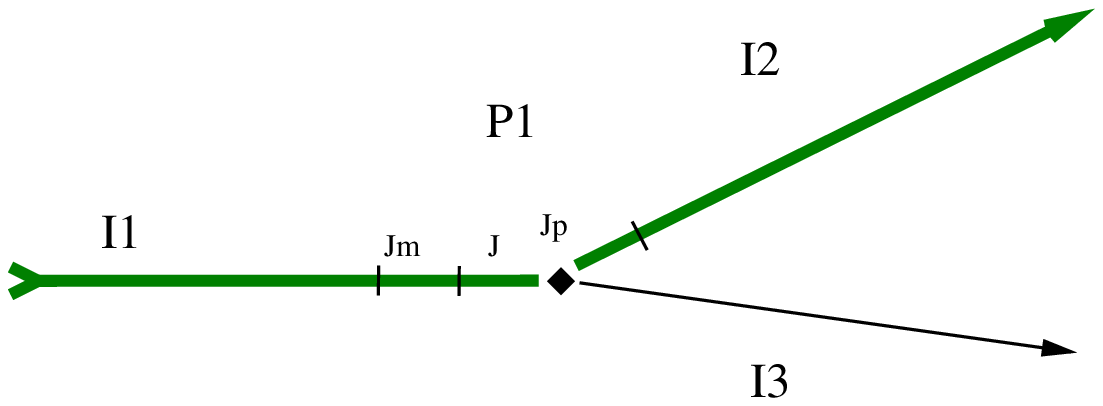}
\includegraphics[width=0.45\textwidth]{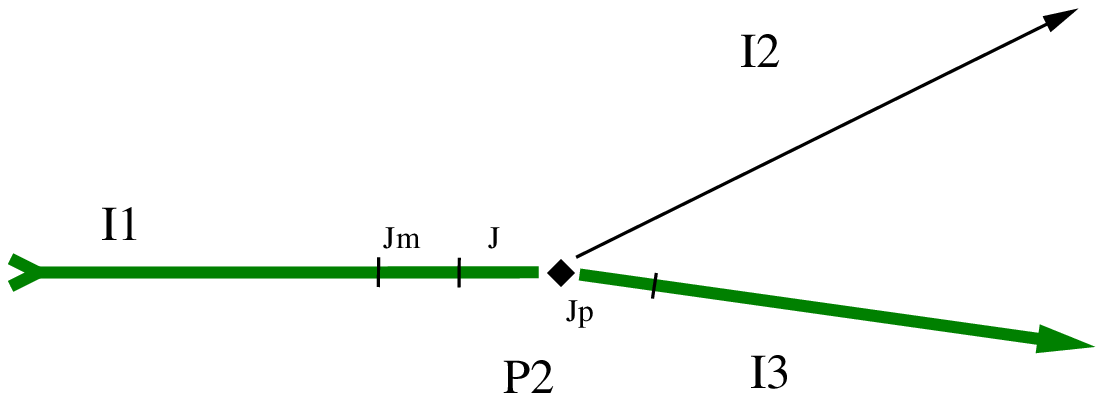}
\end{psfrags}
\end{center}
\caption{A network with 3 arcs and 1 junction, representing a diverge. Path $P^1$ (left) and path $P^2$ (right).}
\label{fig:esempio_network_1in2out}
\end{figure}
We denote by $J$ the cell \emph{just before} the junction. Note that the cells after the junction, namely $J+1$, $J+2$, etc., can refer to one path or the other one, depending on the context.
We have
$$
\omega_{k^1}^{n}=\left\{\begin{array}{ll} 
\mu_{k^1}^{n,1}+\mu_{k^1}^{n,2} & k^1\leq J, \\ 
\mu_{k^1}^{n,1} & k^1> J,\end{array}\right.
\qquad\qquad
\omega_{k^2}^{n}=\left\{\begin{array}{ll} 
\mu_{k^2}^{n,1}+\mu_{k^2}^{n,2} & k^2\leq J, \\ 
\mu_{k^2}^{n,2} & k^2> J,\end{array}\right.
$$
and the scheme (\ref{schema}) becomes equal to \eqref{schema_2in1out}.

Let us now focus on the cell $J$, which is the only one in which the total density could exceed $\rhomax$ (standard properties of the Godunov scheme apply elsewhere).
\begin{theorem}\label{teo:rhoammissibile1in2}
Let the initial densities around the junction be admissible, namely
\begin{equation*}
\mu_{J-1}^{0,1}\leq\rhomax,\quad 
\mu_{J-1}^{0,2}\leq\rhomax,\quad 
(\mu_{J}^{0,1}+\mu_{J}^{0,2})\leq\rhomax,\quad
(\mu_{J+1}^{0,1}+\mu_{J+1}^{0,2})\leq\rhomax.
\end{equation*}
If the \emph{standard} CFL condition holds
\begin{equation}\label{CFLstandard}
\DtsuDx \sup_{\rho\in(0,\rhomax)}|f'(\rho)|\leq 1,
\end{equation}
then
$$
(\mu_{J}^{n,1}+\mu_{J}^{n,2})\leq\rhomax \qquad \forall n.
$$
\end{theorem}

\begin{proof} 
Let us introduce again the auxiliary variable $z^n:=\omega^n_J=\mu^{n,1}_{J}+\mu^{n,2}_{J}$. 
The worst case happens when $\omega_{J-1}^n=\sigma$ (incoming road tries to transfer the maximal flux to cell $J$) and $\mu_{J+1}^{n,1}=\mu_{J+1}^{n,2}=\rhomax$ (no flux from cell $J$ to cell $J+1$). The equation for $z$ is

\begin{multline*}
z^{n+1}=
z^n-\DtsuDx\left(\frac{\mu^{n,1}_J}{z^n} G(z^n,\mu^{n,1}_{J+1}) + \frac{\mu^{n,2}_J}{z^n} G(z^n,\mu^{n,2}_{J+1}) - G(\omega^n_{J-1},z^n) \right)= \\ 
z^n - \frac{\Delta t}{\Delta x}\left(\frac{\mu^{n,1}_{J}}{z^n} G(z^n,\rhomax) + \frac{\mu^{n,2}_{J}}{z^n} G(z^n,\rhomax) - G(\sigma,z^n)\right) 
=z^n + \frac{\Delta t}{\Delta x} G(\sigma,z^n). 
\end{multline*}
The conclusion follows as in Theorem \ref{teo:rhoammissibile2in1}.
\end{proof}

\subsubsection{General junctions} 
Generalizations to junctions with $r_{\textup{inc}}>2$ incoming roads and one outgoing road or one incoming road and $r_{\textup{out}}>2$ outgoing roads are straightforward. In the former case the solution is admissible if the following CFL-like condition holds true
\begin{equation}\label{CFLconN}
r_{\textup{inc}}\DtsuDx\sup_{\rho\in(0,\rhomax)}|f'(\rho)|\leq 1,
\end{equation}
which takes the place of the condition \eqref{CFLcon2}.
In the latter case the solution is admissible if the standard CFL condition \eqref{CFLstandard} holds true.

The general case with $r_{\textup{inc}}>1$ incoming roads and $r_{\textup{out}}>1$ outgoing roads can be also easily solved by means of the ingredients discussed above. To fix the ideas, let us consider the case with $r_{\textup{inc}}=2$ and $r_{\textup{out}}>1$, and focus on the node $J$ just after the junction along one of the outgoing roads. We can now safely extract the subnetwork formed by the considered outgoing road and the two incoming roads. On this subnetwork the situation is very similar to the one discussed in Section \ref{sec:merge}, the only difference is that this time $\mu^{n,p}_{J-1} \neq \omega^n_{J-1}$, $p=1,2$, since along the incoming roads there are vehicles moving to the outgoing roads outside the subnetwork. Since $\frac{\mu^{n,p}_{J-1}}{\omega^n_{J-1}}\leq 1$, we can write (cf.\ \eqref{eq:merge_zn+1})
\begin{equation*}
\begin{split}
z^{n+1}=z^n-\DtsuDx\left(G(z^n,\omega^n_{J+1}) - \frac{\mu^{n,1}_{J-1}}{\omega^n_{J-1}}G(\mu^{n,1}_{J-1},z^n) - \frac{\mu^{n,2}_{J-1}}{\omega^n_{J-1}}G(\mu^{n,2}_{J-1},z^n) \right)\leq 
\\ 
z^n-\DtsuDx\left(G(z^n,\omega^n_{J+1}) - G(\mu^{n,1}_{J-1},z^n) - G(\mu^{n,2}_{J-1},z^n) \right)
\end{split}
\end{equation*}
and then we conclude as in Section \ref{sec:merge}.

Summarizing, condition \eqref{CFLconN} guarantees the admissibility of the solution for any kind of junction.

\section{Similarities to models with buffer}\label{sec:buffer}
In this section we point out some analogies and differences between our model and traffic models with buffer. This also should clarify why in our approach there are no special management of the dynamics at junctions.

Let us first focus on the case of a simple merge, see Section \ref{sec:merge} and \eqref{schema_2in1out}.
Recalling that 
\begin{equation}\label{proprieta_base_G}
G(\rhomeno,\rhopiu)=\min\{G(\rhomeno,\sigma), \ G(\sigma,\rhopiu)\},\qquad \text{for any } \rhomeno,\ \rhopiu\in[0,\rhomax],
\end{equation}
we get that at the cell $J$ just after the junction, the sum $\omega^n_J$ of the two densities $\mu^{n,1}_J$ and $\mu^{n,2}_J$ satisfies the discrete equation
\begin{eqnarray}
\omega_J^{n+1} &=& \omega_J^{n}-\frac{\Delta t}{\Delta x}
(\Gamma_{\textup{out}}-\Gamma_{\textup{in}}), 
\label{2in1_incrocio_J_buffer}
\end{eqnarray}
with
\begin{eqnarray*}
\Gamma_{\textup{out}} &:=& \min\{G(\omega_J^n,\sigma),G(\sigma,\omega^n_{J+1})\},\\
\Gamma_{\textup{in}}  &:=& \min\{G(\mu_{J-1}^{n,1},\sigma),G(\sigma,\omega_J^n)\}+\min\{G(\mu_{J-1}^{n,2},\sigma),G(\sigma,\omega_J^n)\}.
\end{eqnarray*}
The cell $J$ may be seen as an area with an oversize capacity (up to $2f(\sigma)$) which gathers the flows coming from the incoming roads. We may therefore say that the cell $J$ acts as a ``buffer'', i.e.\ a region of size $\Dx$ used to temporarily store vehicles while they are being moved from the incoming roads to the outgoing road. The road then comes back to its original capacity in the cell $J+1$.
In \cite{bressan2014preprint, garavello2012DCDS, garavello2013bookchapt, herty2009NHM} traffic models with buffer are proposed. In such models, the buffer is 0-dimensional, and its load is described by the number of cars $r(t)$ lying at time $t$ in it. The load $r$ varies according to the difference between the inflow and the outflow at the buffer and evolves according to a dedicated ODE. It is also assumed that the maximum number of cars which can enter or exit the buffer per unit of time is a constant parameter.

In the same spirit of \cite{bressan2014preprint, garavello2012DCDS, herty2009NHM}, in our modeling approach the function $\omega_J^n$ may be assumed to describe the evolution of cars densities in the buffer,
i.e.\ using the finite volume approach, we get
$$ \omega_J(t)=\frac{1}{\Dx}\int_0^\Dx z(x,t) dx,$$
where $z(x,t)$ is the densities of cars lying at time $t$ and $x\in J$. 
Equation \eqref{2in1_incrocio_J_buffer} can then be seen as the first-order Euler approximation of the following ODE, 
\begin{equation*}
\frac{d}{dt}\int_0^\Dx z(x,t) dx=\Gamma_{\textup{in}}-\Gamma_{\textup{out}},
\end{equation*}
which expresses nothing but the conservation of vehicles in the buffer, and can be directly compared with \cite[Eq.\ (2.9)]{bressan2014preprint}.
\emph{So, we may say that the numerical scheme \eqref{schema} embeds an approximation of the buffer dynamics}.

Let us now take a look at functions $\Gamma_{\textup{in}}$ and $\Gamma_{\textup{out}}$.
We observe that they are in analogy with the incoming and outgoing fluxes at the buffer proposed in \cite{bressan2014preprint}, while they differ from those proposed in \cite{garavello2012DCDS, herty2009NHM}. In our approach the demand of the buffer differs from the supply of the buffer and they change in time taking into account the current densities lying in the buffer. Indeed, in Section \ref{sec:modified_equation}, we shall define a \textit{modified problem}, where a finite-size buffer appears explicitly. In particular, we will build up a function $h$ that describes in a continuous way the buffer's capacity. 

All those arguments can be extended to any kind of junction. Going back to the general scheme \eqref{schema}, we observe that at any cell $J$ just after the junction along some outgoing road, the densities depend on the non-constant rates ${\mu_{J-1}^{n,p}}/{\omega_{J-1}^{n}}$ before the junction, which gives exactly the fraction of drivers who wish to turn into the considered outgoing road from one of the incoming roads. 
Then, we end up with a \textit{multibuffer} junction, see \cite{bressan2014preprint,garavello2013bookchapt}. Each outgoing road has a buffer and the incoming fluxes at the buffer depend on the percentages of vehicles that from each incoming road turn into the outgoing road. 
%
%
\section{The case of a merge: Numerical solution and modified problem}\label{sec:approfondimento2in1}
In this section we focus on a simple merge, i.e.\ a network with three roads and one junction, with two incoming roads and one outgoing road. On this network two paths $P^1$ and $P^2$ are defined, see Fig.\ \ref{fig:network_2in_1out}.

After the junction, it is convenient dealing with the sum $\omega^n$ of the two densities $\mu^{n,1}$ and $\mu^{n,2}$ rather than the two densities separately. Then, considering again the system \eqref{def_omega_2in1}--\eqref{schema_2in1out}, across the junction we have the following equations:
\begin{eqnarray}
\mu_{J-1}^{n+1,p} &=& \mu_{J-1}^{n,p}-\frac{\Delta t}{\Delta x}\Big(G(\mu_{J-1}^{n,p},\omega_{J}^{n}) - G(\mu_{J-2}^{n,p},\mu_{J-1}^{n,p})\Big), \quad p=1,2,
\label{2in1_allincrocio_Jm1} \\
\omega_J^{n+1} &=& \omega_J^{n}-\frac{\Delta t}{\Delta x}
\Big(G(\omega_J^{n},\omega_{J+1}^{n})- 
\big(G(\mu_{J-1}^{n,1},\omega_{J}^{n})+G(\mu_{J-1}^{n,2},\omega_{J}^{n})\big)\Big).
\label{2in1_allincrocio_J}
\end{eqnarray}

Let us introduce four constants $\ul, \vl, \zc, \wr$ which will be used in the following as initial/boundary conditions before, at, and after the junction. We assume that 
\begin{equation}\label{assumption_uvzw}
\ul,\vl\in(0,\sigma),\qquad \wr\in(\sigma,\rhomax).
\end{equation}
We also consider in the plane $(f(\ul),f(\vl))$ the four regions (A), (B), (B$^\prime$), (C) depicted in Fig.\ \ref{fig:2in1out_regioni} and characterized by the relations
\begin{equation}\label{regions}
\left.
\begin{array}{ll}
(\textup A) & f(\ul)+f(\vl)<f(\wr), \\ [2mm]
(\textup B) & f(\ul)+f(\vl)>f(\wr), \quad f(\vl)<\frac{f(\wr)}{2}, \\ [2mm]
(\textup B^\prime) & f(\ul)+f(\vl)>f(\wr), \quad f(\ul)<\frac{f(\wr)}{2}, \\ [2mm]
(\textup C) &f(\ul)>\frac{f(\wr)}{2},\quad f(\vl)>\frac{f(\wr)}{2}.
\end{array}
\right.
\end{equation}
\begin{figure}[h!]
\begin{center}
\begin{psfrags}
\psfrag{fs}{\hskip-6pt$f(\sigma)$} \psfrag{fb}{\hskip-6pt$f(\beta)$} \psfrag{fa1}{$f(\alphauno)$} \psfrag{fa2}{\hskip-8pt$f(\alphadue)$} \psfrag{fb2}{\hskip-3pt$\frac{f(\beta)}{2}$}\psfrag{(A)}{(A)}\psfrag{(B)}{(B)}\psfrag{(C)}{(B$^\prime$)}\psfrag{(D)}{(C)}
\includegraphics[width=0.5\textwidth]{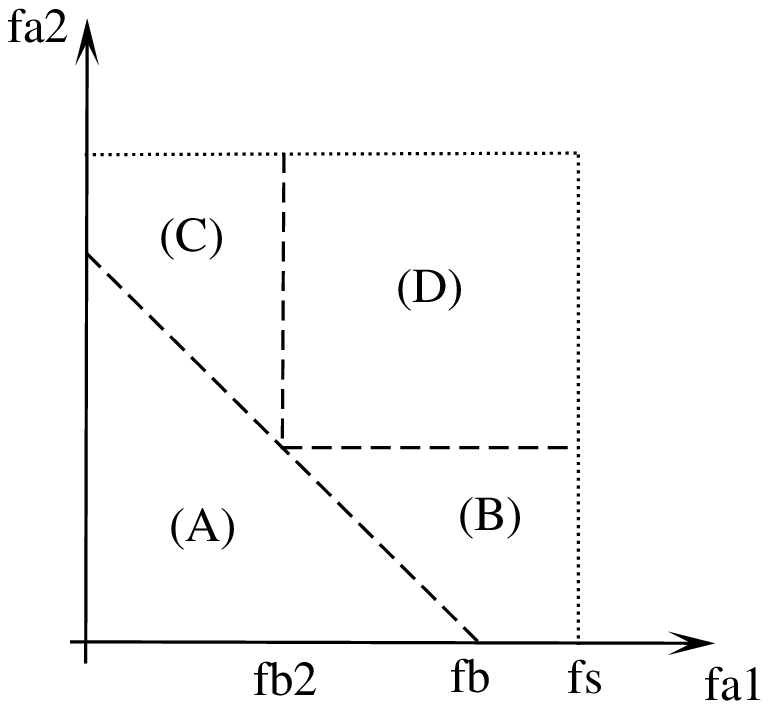}
\end{psfrags}
\end{center}
\caption{Plane $(f(\alphauno),f(\alphadue))$. Regions (A),
(B), (B$^\prime$), and (C) defined in \eqref{regions}.}
\label{fig:2in1out_regioni}
\end{figure}
For symmetry reasons, we will restrict our attention only to regions (A), (B), and (C). Note that the case (A) corresponds to no formation of queues, the case (B) corresponds to the formation of a queue along $P^1$, and the case (C) corresponds to the formation of two queues along $P^1$ and $P^2$.
In addition, we assume that
\begin{equation}\label{zc=sigma}
\zc=\sigma.
\end{equation}
Assumptions \eqref{assumption_uvzw} and \eqref{zc=sigma} are mainly technical and could be relaxed. They are taken to ease the proofs of the next Theorems \ref{teo:puntistazionari} and \ref{teo:RPteorico}, where the continuous and the discrete problem are analyzed and compared.
Indeed, we still cover all the interesting cases (free and congested state, queue formation) while limiting the number of cases to be studied, otherwise astronomical. 

Let us also introduce the following notation, which will be used through the paper.
\begin{notation}\label{not:tau_di}
Let $f$ be a flux function satisfying \eqref{proprieta_f}. For any density $\rho\in[0,\rhomax]$, we denote by $\rho^\compl$ the unique density such that
\begin{equation}\label{tau}
\rho^\compl\neq\rho \quad\textrm{ and }\quad f(\rho^\compl)=f(\rho).
\end{equation}
\end{notation}
For example, if $f(\rho)=\rho(1-\rho)$ with $\rhomax=1$, then $\rho^\compl=1-\rho$.

\subsection{Asymptotic numerical solution}\label{sec:Riemann_numerico_2in1}
Let us focus on the behaviour of the scheme at the junction, considering only the two cells $J-1$ and the cell $J$, see  \eqref{2in1_allincrocio_Jm1}--\eqref{2in1_allincrocio_J}. We set the following boundary conditions:
\begin{equation}\label{CB_2in1out}
\forall n\geq 0,\qquad \mu_{J-2}^{n,1}=\alphauno,\qquad \mu_{J-2}^{n,2}=\alphadue, \qquad \mu_{J+1}^{n,1}=\beta^1, \qquad \mu_{J+1}^{n,2}=\beta^2,
\end{equation}
with $\wr:=\wr^1+\wr^2$.
The next result concerns the stationary solutions of the scheme.
\begin{theorem}\label{teo:puntistazionari}
Consider the scheme at cells $J-1$ and $J$, namely \eqref{2in1_allincrocio_Jm1}--\eqref{2in1_allincrocio_J}, with boundary conditions \eqref{CB_2in1out} satisfying \eqref{assumption_uvzw}. 
Then, the unique stationary solution $(\bar\mu^{1}_{J-1},\bar\mu^{2}_{J-1},\bar\omega_J=\bar\mu^{1}_J+\bar\mu^{2}_J)$ is determined by the following conditions:
\begin{itemize}
\item \textup{CASE (A)}
\begin{equation}\label{puntostazionarioA}
\left\{
\begin{array}{l}
\bar\mu^{1}_{J-1}=\alphauno \\ [1mm]
\bar\mu^{2}_{J-1}=\alphadue\\ [1mm]
\bar\omega_J<\sigma, \textrm{ s.t. } f(\bar\omega_J)=f(\alphauno)+f(\alphadue).
\end{array}
\right.
\end{equation}
\item \textup{CASE (B)}
\begin{equation}\label{puntostazionarioB}
\left\{
\begin{array}{l}
\bar\mu^{1}_{J-1}>\sigma, \textrm{ s.t. } f(\bar\mu^{1}_{J-1})=f(\beta)-f(\alphadue)\\ [1mm]
\bar\mu^{2}_{J-1}=\alphadue\\ [1mm]
\bar\omega_{J}>\sigma, \textrm{ s.t. } f(\bar\omega_{J})=f(\beta)-f(\alphadue).
\end{array}
\right.
\end{equation}
\item \textup{CASE (C)}
\begin{equation}\label{puntostazionarioD}
\bar\mu^{1}_{J-1}=\bar\mu^{2}_{J-1}=\bar\omega_{J}>\sigma, \textrm{ s.t. } 
f(\bar\mu^{1}_{J-1})=f(\bar\mu^{2}_{J-1})=f(\bar\omega_{J})=\frac{f(\beta)}{2}.
\end{equation}
\end{itemize}
Moreover, one can compute $\bar\mu_J^1$ and $\bar\mu_J^2$ by means of $\bar\omega_J$ as follows:
\begin{itemize}
\item \textup{CASE (A)}$\quad$ $\bar\mu_J^{1}=\frac{f(\alphauno)}{f(\alphauno)+f(\alphadue)}\bar\omega_J,\quad \bar\mu_J^2=\bar\omega_J-\bar\mu_J^1$. \\ [1mm]
\item \textup{CASE (B)}$\quad$ $\bar\mu_J^{1}=\frac{f(\beta)-f(\alphadue)}{f(\beta)}\bar\omega_J,\quad \bar\mu_J^2=\bar\omega_J-\bar\mu_J^1$. \\ [1mm]
\item \textup{CASE (C)}$\quad$ $\bar\mu^{1}_{J}=\bar\mu^{2}_{J}=\bar\omega_{J}/2$.
\end{itemize}
\end{theorem}
\noindent The proof is given in the Appendix \ref{app:A1}.
\begin{remark}
As already noted in \cite{bretti2013DCDS-S}, when queues are formed along both incoming roads, the density split in two equal values, regardless the ratio between the boundary conditions $\ul$ and $\vl$. In the classical approaches based on the LWR model \cite{piccolibook}, this effect is achieved by maximizing the flux through the junction and setting the priorities coefficients to $\frac12$ (incoming fluxes are equidistributed). See \cite{bretti2013DCDS-S} for a detailed numerical comparison between the multi-path scheme and the Godunov scheme for the classical model.
\end{remark}
The next result concerns the stability of the stationary points exhibited in Theorem \ref{teo:puntistazionari}.
\begin{theorem}\label{teo:stability}
Let the CFL condition \eqref{CFLstandard} hold true. Then, the three stationary points \eqref{puntostazionarioA},\eqref{puntostazionarioB},\eqref{puntostazionarioD} are locally asymptotically stable under the respective conditions. 
\end{theorem}
\noindent The proof is given in the Appendix \ref{app:A2}.
Note that the CFL-like condition \eqref{CFLcon2} needed for the applicability of the scheme implies the standard CFL condition \eqref{CFLstandard}.
\begin{remark}
Borderline cases are relatively easy to study. The case $\beta=\sigma$ is essentially analogous to the case $\beta>\sigma$ and all the results described above are still valid. Instead, a more complicated situation arises by assuming that the boundary conditions $\alphauno$, $\alphadue$, and $\beta=\beta^1+\beta^2$ are chosen in such a way that the point $(f(\alphauno), f(\alphadue))$ lies at the boundaries of all four regions depicted in Fig.\ \ref{fig:2in1out_regioni}. In that case, infinite stationary points are present. Indeed, if \eqref{assumption_uvzw} holds true and
$$
f(\alphauno)=f(\alphadue)=\frac{f(\beta)}{2}
$$
we necessarily have $\alphauno=\alphadue$ and all the points $(\bar\mu^{1}_{J-1},\bar\mu^{2}_{J-1},\bar\mu^{1}_J,\bar\mu^{2}_J)$
such that
$$
\bar\mu_{J-1}^1=\alphauno, \quad
\bar\mu_{J-1}^2=\alphauno, \quad
\bar\mu_{J}^1=\bar\mu_J^2=\frac{\bar\omega_J}{2}, \textrm{ with } \bar\omega_J\in[\beta^\compl,\alphauno^\compl]
$$ are stationary, but not in general stable.
\end{remark}

\subsection{The modified equation}\label{sec:modified_equation}
In this section we solve an inverse problem, in analogy with the ``modified equation'' in the sense of LeVeque \cite[Sect.\,11.1]{levequebook}. More precisely, we find a new set of equations, compatible with the classical theory of traffic flow on networks (where a single equation is solved on any arc of the graph) \cite{piccolibook}, such that the solution of the Riemann problem associated to those equations coincide with the solution of the numerical scheme exhibited in Section \ref{sec:Riemann_numerico_2in1}. 

In order to introduce the modified equation, it is convenient to slightly reformulate the problem assuming roads to be half lines and changing notations. We denote by $u(x,t):(-\infty,0]\times[0,+\infty)\to [0,\rhomax]$ and $v(x,t): (-\infty,0]\times[0,+\infty)\to [0,\rhomax]$ the densities along the first and second incoming roads, respectively, by $z(x,t):[0,\Dx]\times[0,+\infty)\to [0,\rhomax]$ the density inside the cell $J$ just after the junction, and by $w(x,t):[\Dx,+\infty)\times[0,+\infty)\to [0,\rhomax]$ the density along the outgoing road.
Then we face the following problem
\begin{equation}\label{pb_riemann_gen}
\left\{
\begin{array}{ll}
\frac{\partial}{\partial t}u + \frac{\partial}{\partial x}f(u)=0, & x<0, \quad t>0,\\ [2mm]
\frac{\partial}{\partial t}v + \frac{\partial}{\partial x}f(v)=0, & x<0, \quad t>0,\\ [2mm]
\frac{\partial}{\partial t}z + \frac{\partial}{\partial x}h(z,x,t;u,v,w)=0, & 0<x<\Dx, \quad t>0, \\ [2mm]
\frac{\partial}{\partial t}w + \frac{\partial}{\partial x}f(w)=0, & x>\Dx, \quad t>0,
\end{array}
\right.
\end{equation}
where
\begin{equation}\label{def:h(x,t)}
h(z,x,t; u,v,w):=
f(z) + C(x,t; u^-(t), v^-(t), w^+(t)),
\end{equation}
the function $0\leq C \leq f(\sigma)$ will be chosen in the following, and we have defined $u^-(t):=u(0-,t)$, $v^-(t):=v(0-,t)$, and $w^+(t):=w(\Dx+,t)$.

Let us comment the choice of the flux in \eqref{def:h(x,t)}. It is just a translation of the original flux $f$, since it has the form $h=f+C$, where $C$ depends on space, time, and the time-varying traces of the densities along the roads connected to the junction. Note that we do not require $h=0$ for $z=0,\rhomax$, then $h$ is not in general a physical flux. 
The choice of $h$ translates the fact that \textit{the capacity of the junction is variable in space and time, and it is automatically adjusted on the basis of the demand of the incoming roads and the supply of the outgoing road}. After the junction, the road comes back to its original size by means of a bottleneck, see Fig.\ \ref{fig:pbmodificato_2in1}. 
\begin{figure}[h!]
\begin{center}
\begin{psfrags}
\psfrag{rho1}{$u$}
\psfrag{rho2}{$\ \ v$}
\psfrag{rhoj}{$\ z$}
\psfrag{rho3}{$w$}
\psfrag{Dx}{$\!\!\!\!\!\!\!\!\leftarrow\Dx\rightarrow$}
\includegraphics[width=0.3\textwidth]{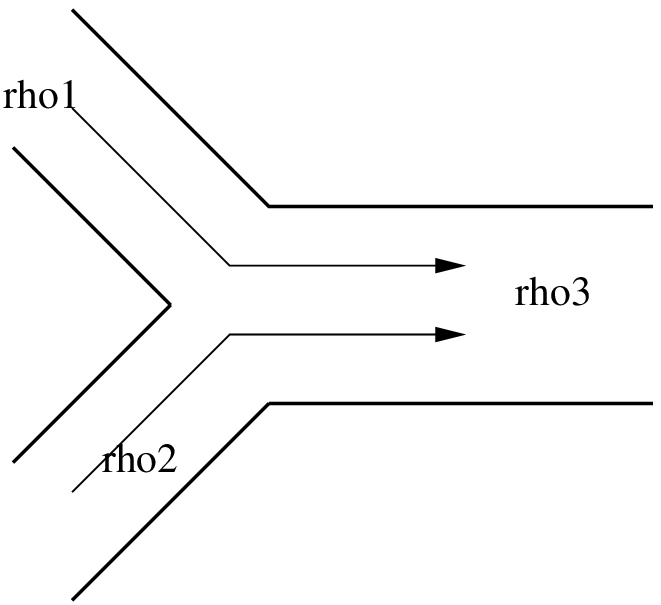}\qquad\qquad
\includegraphics[width=0.35\textwidth]{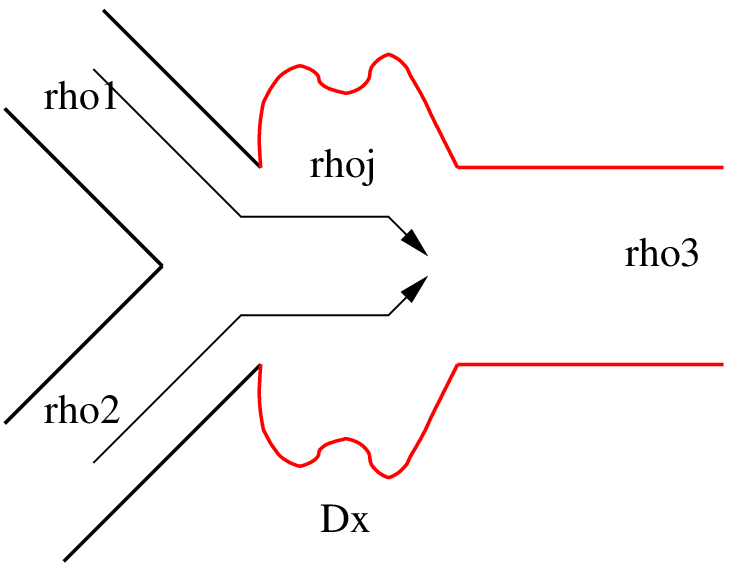}
\end{psfrags}
\end{center}
\caption{Original problem (left) and modified problem (right).}
\label{fig:pbmodificato_2in1}
\end{figure}

To be coherent with the numerical investigation, we set the initial data as
\begin{equation}\label{pb_riemann_gen_init_data}
u(x,0)\equiv\ul, \quad
v(x,0)\equiv\vl, \quad
z(x,0)\equiv\zc, \quad 
w(x,0)\equiv\wr,
\end{equation}
with $\ul,\vl,\zc,\wr \in [0,\rhomax]$ and satisfying assumptions \eqref{assumption_uvzw} and \eqref{zc=sigma}.

In the following we show that, for a particular choice of $C$, the solution to the Riemann problem associated to equations \eqref{pb_riemann_gen}--\eqref{pb_riemann_gen_init_data} equals the solution of the numerical scheme around the junction, namely the scheme described in Section \ref{sec:scheme} is \textit{consistent} with equations \eqref{pb_riemann_gen}--\eqref{def:h(x,t)} for a special choice of $h$. 

We define a solution to the Riemann problem as follows (cfr.\ \cite[Def.\ 4.4.1]{piccolibook}):
\begin{definition}\label{def:J_riemann_sol} We say that a set of smooth states 
determine a solution to the Riemann problem \eqref{pb_riemann_gen}--\eqref{pb_riemann_gen_init_data} if
\begin{itemize}
\item[(R1)] The waves generated between two adjacent constant states at $x=0-$ and $x=\Dx-$ have negative speed.
\item[(R2)] The waves generated between two adjacent constant states on $x=0+$ and $x=\Dx+$ have positive speed.
\item[(R3)] (conservation of flux at $x=0$)
For all $t\geq 0$, 
\begin{equation}\label{boundary_0}
f(u(0-,t))+f(v(0-,t)) = h(z(0+,t),0+,t).
\end{equation}
\item[(R4)] (conservation of flux at $x=\Dx$) 
For all $t\geq 0$,
\begin{equation}\label{boundary_Dx}
h(z(\Dx-,t),\Dx-,t) = f(w(\Dx+,t))
\end{equation}
\end{itemize}
\end{definition}
\noindent (see Appendix \ref{sec:Riemann_and_half-Riemann_intro} for basic notions about Riemann and half-Riemann problems). 
Generally, this definition does not select a unique solution to the problem \eqref{pb_riemann_gen}--\eqref{pb_riemann_gen_init_data}. 
To fix this, we select particular fluxes at $x=0,\Dx$ somehow ``suggested'' by the numerical scheme, and we study the corresponding evolution of the density. Doing this, we verify that the Riemann problem is not degenerate (waves do not rebound forever) and we compute the constant asymptotic solution for $t\to +\infty$.

We also assume that at the initial time $t=0$, for any $x\in[0,\Dx]$, we have
\begin{equation}\label{def:hc_primavolta}
h(z,x,0)=h_c(z):=f(z)+f(\sigma),
\end{equation}
(or, equivalently, $C(x,0)\equiv f(\sigma)$), meaning that the maximal capacity of the junction $[0,\Delta x]$ doubles the maximal capacity of the other roads. Note that assumptions \eqref{zc=sigma} and \eqref{def:hc_primavolta} translate the fact that at the initial time the junction area itself is not responsible for a possible congestion, while, if any, this will be due to the choice of $\wr$.
\begin{theorem}\label{teo:RPteorico}
Consider the problem \eqref{pb_riemann_gen}--\eqref{pb_riemann_gen_init_data}, where at any time $t>0$ the choice of fluxes at junction is given by:
\begin{equation}\label{choice_at_junction_in}
\begin{split}
\text{Flux at $x=0$:}  \qquad
\min\{G_f(u(0-,t),\sigma),G_f(\sigma ,z(0+,t))\} + \\  
\min\{G_f(v(0-,t),\sigma),G_f(\sigma ,z(0+,t))\}
\end{split}
\end{equation}
\begin{eqnarray}\label{choice_at_junction_out}
\textrm{\ \ \ \ \ \ \ \ Flux at $x=\Dx$:} \quad \ \min\{G_h(z(\Dx-,t),\sigma), G_f(\sigma,w(\Dx+,t))\}
\end{eqnarray}
where $G_\cdot(\cdot,\cdot)$ is the Godunov flux defined in \eqref{GodunovFlux} and $h$ has the form \eqref{def:h(x,t)}. Then there exists a function $C$ such that the following holds: (i)
For every initial data $\ul,\vl,\zc,\wr$ satisfying \eqref{assumption_uvzw} and \eqref{zc=sigma} there exists one solution (in the sense of Definition \ref{def:J_riemann_sol}) of the problem. (ii) By assuming the initial data to be in the regions (A), (B), (C) defined in \eqref{regions}, the following constant quadruplets $(\bar u,\bar v,\bar z,\bar w)$ are, respectively, asymptotic stationary solutions of the problem.
\begin{itemize}
\item \textup{CASE (A)}
\begin{equation*}
\left\{
\begin{array}{l}
\bar u=\ul \\
\bar v=\vl \\
\bar z<\sigma, \textrm{ s.t. } f(\bar z)=f(\ul)+f(\vl) \\
\bar w<\sigma, \textrm{ s.t. } f(\bar w)=f(\ul)+f(\vl).
\end{array}
\right.
\end{equation*}
\item \textup{CASE (B)}
\begin{equation*}
\left\{
\begin{array}{l}
\bar u > \sigma,  \textrm{ s.t. } f(\bar u)=f(\wr)-f(\vl)\\
\bar v=\vl \\
\bar z > \sigma,  \textrm{ s.t. } f(\bar z)=f(\wr)-f(\vl) \\
\bar w=\wr.
\end{array}
\right.
\end{equation*}
\item \textup{CASE (C)}
\begin{equation*}
\left\{
\begin{array}{l}
\bar u=\bar v = \bar z >\sigma \textrm{ s.t. } f(\bar u)=f(\bar v)=f(\bar z)=\frac{f(\wr)}{2} \\
\bar w=\wr.
\end{array}
\right.
\end{equation*}
\end{itemize}
\end{theorem}
\noindent The proof is given in the Appendix \ref{app:A3}, together with the exact expression of $C$ (Remark \ref{rem:htotale}).
\begin{remark}
The main result of the paper comes from the comparison between Theorems \ref{teo:puntistazionari} and \ref{teo:RPteorico}.
Clearly we cannot compare exactly the numerical and the theoretical solution because the first one is composed by a triplet $(\bar\mu^1_{J-1},\bar\mu^2_{J-1},\bar\omega_{J}=\bar\mu^1_{J}+\bar\mu^2_{J})$ while the second one by a quadruplet $(\bar u,\bar v,\bar z,\bar w)$. The reason for this is mainly technical, i.e.\ we chose to have only three unknowns in the numerical setting to make the proof doable.
Anyway, in the case (A) we have the correspondence $(\bar u,\bar v,\bar z=\bar w)=(\bar \mu^1_{J-1},\bar \mu^2_{J-1},\bar \omega_{J})$, and in the case (B),(C) we have the correspondence $(\bar u,\bar v,\bar z)=(\bar \mu^1_{J-1},\bar \mu^2_{J-1},\bar \omega_{J})$.
\end{remark}

\subsection{Numerical tests}\label{numerics2in1}
In this section we present three numerical tests in the case of a merge, in order to confirm experimentally the results described in Sections \ref{sec:Riemann_numerico_2in1} and \ref{sec:modified_equation}. A more complete numerical study of the multi-path model presented here can be found in \cite{bretti2013DCDS-S}.
We assume that each arc has the same length, equal to 1, then each path has length equal to 2. We also assume that the flux has the classical form $f(\rho):=\rho(1-\rho)$, so that $\rhomax=1$ and $\sigma=0.5$. We divide each arc in 25 cells (then the junction is found at $J=26$ along each path) and we impose Dirichlet boundary conditions at the beginning and the end of each path, i.e.\ at points $0^{(1)}$, $0^{(2)}$, $2^{(1)}$, and $2^{(2)}$, see Fig.\ \ref{fig:network_2in_1out}. Note that the points $2^{(1)}$ and $2^{(2)}$ correspond to the same physical point. \textit{Mutatis mutandis}, these boundary conditions have the same role of $\ul$, $\vl$, $\wr^1$, and $\wr^2$ in \eqref{CB_2in1out}. We assume roads are empty at the initial time and we look for the stationary solution obtained for $t\to\infty$.

\medskip

\textit{Test 1.} We consider the case (A) in \eqref{regions} (no formation of queue along the incoming roads). For any $t>0$ we set
$$
\mu^1(0^{(1)},t)=0.1, \quad
\mu^2(0^{(2)},t)=0.15, \quad
\mu^1(2^{(1)},t)=0.3, \quad
\mu^2(2^{(2)},t)=0.3.
$$
In Fig.\ \ref{fig:numtest(A)} we report the stationary solution for $\mu^1$ along $P^1$, $\mu^2$ along $P^2$ and the total density $\omega$ along $P^1$. 
\begin{figure}[h!]
\begin{center}
\includegraphics[width=0.5\textwidth]{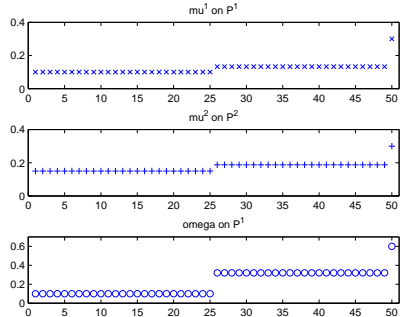}
\end{center}
\caption{Test 1. Stationary solution for $\mu^1$ along $P^1$, $\mu^2$ along $P^2$ and the total density $\omega$ along $P^1$.}
\label{fig:numtest(A)}
\end{figure}
As expected, the right boundary condition has no effect on the solution. After the junction, the total density $\omega$ is equal to $\bar \lambda \approx 0.3197<\sigma$ and it is such that $f(\bar \lambda)=f(0.1)+f(0.15)$. The value $\bar \lambda$ corresponds to $\bar\omega_J$ in Theorem \ref{teo:puntistazionari} and to $\bar z=\bar w$ in Theorem \ref{teo:RPteorico}.
Correctly, the values of the single densities $\mu^1$ and $\mu^2$ after the junction are equal to
$$
\frac{f(0.1)}{f(0.1)+f(0.15)}\bar \lambda \approx 0.1323
\quad \textrm{ and } \quad
\frac{f(0.15)}{f(0.1)+f(0.15)}\bar \lambda \approx 0.1874.
$$ 

\medskip

\textit{Test 2.} We consider the case (B) in \eqref{regions} (formation of a queue along the first incoming road). For any $t>0$ we set
$$
\mu^1(0^{(1)},t)=0.3, \quad
\mu^2(0^{(2)},t)=0.1, \quad
\mu^1(2^{(1)},t)=0.35, \quad
\mu^2(2^{(2)},t)=0.25.
$$
In Fig.\ \ref{fig:numtest(B)} we report the stationary solution for $\mu^1$ along $P^1$, $\mu^2$ along $P^2$ and the total density $\omega$ along $P^1$. 
\begin{figure}[h!]
\begin{center}
\includegraphics[width=0.5\textwidth]{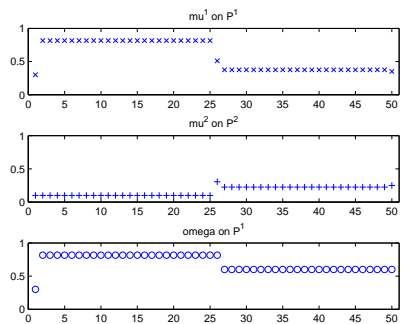}
\end{center}
\caption{Test 2. Stationary solution for $\mu^1$ along $P^1$, $\mu^2$ along $P^2$ and the total density $\omega$ along $P^1$.}
\label{fig:numtest(B)}
\end{figure}
Here, the right boundary conditions affect the solution, but only by means of their sum $0.6=0.35+0.25$, while the left boundary conditions have effect only along $P^2$. 
Correctly, the values of the single densities $\mu^1$ and $\mu^2$ before the junction are equal to $\bar\lambda\approx 0.8162>\sigma$ with $f(\bar\lambda)=f(0.6)-f(0.1)$ and $0.1$, respectively.
The value $\bar \lambda$ corresponds to $\bar\mu^1_{J-1}=\bar\omega_J$ in Theorem \ref{teo:puntistazionari} and to $\bar u=\bar z$ in Theorem \ref{teo:RPteorico}.
Looking at the single densities, at the cell $J$ we find two isolated values equal to 
$$
\frac{f(0.6)-f(0.1)}{f(0.6)}\bar\lambda\approx 0.5101
\quad \textrm{ and } \quad
\bar\lambda-0.5101 \approx 0.3061,
$$ 
corresponding respectively to $\bar\mu^1_{J}$ and $\bar\mu^2_{J}$ in Theorem \ref{teo:puntistazionari}.

Surprisingly, if we look at the total density $\omega$ we see that \textit{the cell $J$ plays the role of the last cell of the incoming roads even if it defined as the first cell of the outgoing road}. In other words, the scheme shifts the junction to one cell to the right. This happens also in the case (C) but does not happen in the case (A). Indeed, in order to perceive the presence of the junction, the multi-path scheme needs an additional cell w.r.t.\ the classical approaches \cite{piccolibook}, where instead the dynamics at junctions is given instantaneously by an external procedure.

\medskip

\textit{Test 3.} We consider the case (C) in \eqref{regions} (formation of two queues along the incoming roads). For any $t>0$ we set
$$
\mu^1(0^{(1)},t)=0.2, \quad
\mu^2(0^{(2)},t)=0.3, \quad
\mu^1(2^{(1)},t)=0.3, \quad
\mu^2(2^{(2)},t)=0.5.
$$
In Fig.\ \ref{fig:numtest(C)} we report the stationary solution for $\mu^1$ along $P^1$, $\mu^2$ along $P^2$ and the total density $\omega$ along $P^1$. 
\begin{figure}[h!]
\begin{center}
\includegraphics[width=0.5\textwidth]{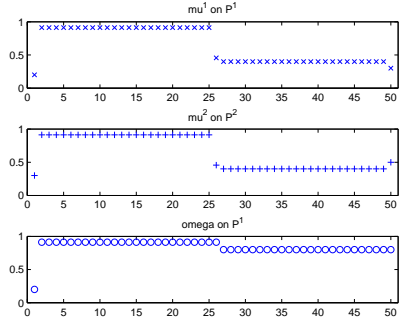}
\end{center}
\caption{Test 3. Stationary solution for $\mu^1$ along $P^1$, $\mu^2$ along $P^2$ and the total density $\omega$ along $P^1$.}
\label{fig:numtest(C)}
\end{figure}
Here, the right boundary conditions affect the solution, but only by means of their sum $0.8=0.3+0.5$, while the left boundary conditions have no effect. 
The densities split equally along the two paths, regardless the fact that the boundary conditions are not equal. We refer the reader to \cite{bretti2013DCDS-S} for a comparison with the classical models and the relationship with priority coefficients. 
Before the junction, the densities are equal to $\bar \lambda \approx 0.9123>\sigma$ with $f(\bar \lambda)=\frac{f(0.8)}{2}$. The value $\bar \lambda$ corresponds to $\bar\mu^1_{J-1}=\bar\mu^2_{J-1}$ in Theorem \ref{teo:puntistazionari} and to $\bar u=\bar v$ in Theorem \ref{teo:RPteorico}. 
At the cell $J$ we find an isolated value equal to $\frac{\bar \lambda}{2}$, corresponding to $\bar\mu^1_{J}=\bar\mu^2_{J}$ in Theorem \ref{teo:puntistazionari}. We refer again to \cite{bretti2013DCDS-S} for a numerical study of the density evolution at this cell.

\section{Conclusions}\label{sec:conclusions}
In this paper we investigated the properties of a Godunov-based numerical scheme for a first-order macroscopic model describing traffic flow on networks. The scheme can be implemented in minutes since it does not requires additional procedures to manage the solution at junctions. Numerical and computational details have been left to the ``twin'' paper \cite{bretti2013DCDS-S}.

Despite the simplicity of the numerical discretization, the method shows many interesting properties. In the following we summarize the main features of our approach.

\begin{itemize}
\item 
Let us recall that the quantities $G(\rhomeno,\sigma)$ and $G(\sigma,\rhopiu)$ appearing in \eqref{proprieta_base_G} match exactly the definitions of the maximum incoming flux and of the maximum outgoing flux, respectively, that can be obtained on each road. Namely, for all incoming roads, 
\begin{equation*}\label{def:gammamaxin}
G(\rho,\sigma)= 
\left\{\begin{array}{ll}
f(\rho), & \mbox{ if } \rho\in[0,\sigma],\\
f(\sigma), & \mbox{ if } \rho\in(\sigma,\rhomax],
\end{array}\right.
\end{equation*}
and, for all outgoing roads, 
\begin{equation*}\label{def:gammamaxout}
G(\sigma,\rho)= 
\left\{\begin{array}{ll}
f(\sigma), & \mbox{ if } \rho\in[0,\sigma],\\
f(\rho), & \mbox{ if } \rho\in(\sigma,\rhomax].
\end{array}\right.
\end{equation*}
They also correspond to the demand and supply functions used in \cite{herty2009NHM,lebacque1996proc}. 
As a consequence, the multi-path scheme selects automatically a solution at junctions that maximizes the flow along each path (user optimum). The scheme \textit{does not compute in general the maximal flow that could possibly be transferred over the node} (global optimum), as it happens in more standard approaches. See \cite{bretti2013DCDS-S} for a numerical evidence of this fact and \cite[Sect.\ 3.1]{tampere2011TRB} for a discussion on this point.

\item In the challenging case of a merge, we proved that the proposed numerical scheme is consistent with the Riemann problem \eqref{pb_riemann_gen}--\eqref{pb_riemann_gen_init_data}. As a consequence, the numerical approximation automatically assigns to the junction a {\it finite spatial dimension} and its density evolution is managed by a special flux function $h$, as defined in \eqref{def:h(x,t)}. Specifically, the function $h$ balances the incoming flows and the outgoing flow according to the fact that each population wants to maximize its own flow. It is able to widen temporarily the spaced-junction capacity, allowing the passage of vehicles and making the junction act as a ``buffer'', see Section \ref{sec:buffer}. The exact expression of the space-time-dependent flux $h$ can be found in the proof of Theorem \ref{teo:RPteorico} (see Remark \ref{rem:htotale} in Appendix \ref{app:A3}).

\item The multi-path model, together with the proposed discretization, fulfills all the 7 requirements for junction models outlined in \cite[Sect.\ 3.1]{tampere2011TRB}. In particular, it is generally applicable irrespectively of the number of incoming and outgoing roads, the traffic never flows backwards and all flows are non-negative,  vehicles are conserved, and turning fractions are preserved.

\item Finally we have shown that the standard CFL condition \eqref{CFLstandard} is not in general sufficient to ensure that the solution of the numerical scheme \eqref{schema} is admissible at junctions. Instead, it is required a stronger condition which depends on the number of incoming roads at the junction. This condition is given, in the case of a junction with $r_{\textup{inc}}$ incoming roads, by \eqref{CFLconN}.
\end{itemize}

\appendix
\section{Technical proofs}\label{app:A}

\subsection{Proof of Theorem \ref{teo:puntistazionari}.}\label{app:A1}
\begin{proof}
To simplify the notations, let us introduce the auxiliary variables
\begin{equation}\label{def:xyz}
x^n:=\mu^{n,1}_{J-1},  \qquad
y^n:=\mu^{n,2}_{J-1}, \qquad
z^n:=\omega_J^n=\mu^{n,1}_{J}+\mu^{n,2}_{J}.
\end{equation}
At the node $J$ it is convenient handling the sum of the two densities rather than the two densities separately. The densities $\mu^{n,1}_{J}$ and $\mu^{n,2}_{J}$ can be univocally found \textit{a posteriori}. In the proof we employ the notation $\compl$ defined by \eqref{tau}.
The discrete dynamical system \eqref{2in1_allincrocio_Jm1}--\eqref{2in1_allincrocio_J} with boundary conditions \eqref{CB_2in1out} is written as
\begin{equation}\tag{DS}\label{DS}
\left\{
\begin{array}{l}
x^{n+1}=x^n-\DtsuDx\big(G(x^n,z^n)-G(\alphauno,x^n)\big)\\ [2mm]
y^{n+1}=y^n-\DtsuDx\big(G(y^n,z^n)-G(\alphadue,y^n)\big)\\ [2mm]
z^{n+1}=z^n-\DtsuDx\big(G(z^n,\beta)-G(x^n,z^n)-G(y^n,z^n)\big).
\end{array}
\right. 
\end{equation}
Therefore, a stationary point $(x,y,z)$ for (\ref{DS}) must satisfy
\begin{eqnarray}
G(x,z)&=&G(\alphauno,x) \label{DS1} \\
G(y,z)&=&G(\alphadue,y) \label{DS2} \\
G(z,\beta)&=&G(x,z)+G(y,z)=G(\alphauno,x)+G(\alphadue,y). \label{DS3} 
\end{eqnarray}

Now we consider the three cases.
\begin{itemize}
\item CASE (A) 
\begin{equation}\label{(A)Riemann_proof_numerica}
\boxed{f(\ul)+f(\vl) < f(\wr)}
\end{equation}
\begin{itemize}
\item CASE A.1
\begin{equation}\label{11i}
\left\{
\begin{array}{l}
G(\alphauno,x)=f(\alphauno) \\ [2mm]
G(\alphadue,y)=f(\alphadue)
\end{array}
\right. 
\end{equation}

The quantity $G(z,\beta)$ can be in principle equal to $f(z)$, $f(\beta)$, or $f(\sigma)$. But \eqref{assumption_uvzw} $\Rightarrow (\wr>\sigma) \Rightarrow
G(z,\beta)\neq f(\sigma)$.
Moreover, if $G(z,\beta)=f(\beta)$, by \eqref{DS3} and \eqref{11i} we would have $f(\ul)+f(\vl)=f(\wr)$, which is in contradiction with \eqref{(A)Riemann_proof_numerica}. Then we must have 
\begin{equation}\label{g_casoA1}
G(z,\beta)=f(z).
\end{equation}
Now, \eqref{assumption_uvzw}+\eqref{g_casoA1} $\Rightarrow (z=\beta$) or ($z<\beta$ and $f(z)\leq f(\beta)$).\\
But $(z=\beta)$ $\stackrel{\eqref{DS3}+\eqref{11i}+\eqref{g_casoA1}}{\Rightarrow}$ $f(\alphauno)+f(\alphadue)=f(z)=f(\beta)$, which is impossible because of \eqref{(A)Riemann_proof_numerica}. 
Then, we must have $z<\beta$ and $f(z)\leq f(\beta)$, which implies
\begin{equation}\label{113i}
z\leq \beta^\compl<\sigma.
\end{equation}
By $f(\alphauno)+f(\alphadue)=f(z)$ and \eqref{assumption_uvzw} we also get
\begin{equation}\label{113ii}
z>\alphauno \textrm{ and } z>\alphadue.
\end{equation}
By \eqref{DS1}+\eqref{11i} we get $G(x,z)=f(\alphauno)$. Since $\ul<\sigma$, this can be true only if $x=\alphauno$ or $x=\alphauno^\compl$ or $z=\alphauno$ or $z=\alphauno^\compl$. But
\eqref{113ii} $\Rightarrow$ $z\neq\alphauno$, and 
\eqref{113i} $\Rightarrow$ $z\neq\alphauno^\compl$.
Moreover, ($x=\alphauno^\compl$) $\stackrel{\eqref{assumption_uvzw}}{\Rightarrow}$ ($x>\sigma$) $\stackrel{\eqref{113i}}{\Rightarrow}$ ($G(x,z)=f(\sigma)$) $\stackrel{\eqref{DS1}}{\Rightarrow}$ ($G(\alphauno,x)=f(\sigma)$), which is impossible because $\alphauno<\sigma$.
Then, $x=\alphauno$ is the only valid candidate.

Following same reasoning, we get $y=\alphadue$. Indeed, a simple verification shows that $x=\alphauno$, $y=\alphadue$ and $z<\sigma$ such that
\begin{equation}\label{113iii}
f(z)=f(\alphauno)+f(\alphadue)
\end{equation} 
is a stationary point.

Once $z$ is found, it is easy to find $z^1:=\mu_J^1$ and $z^2:=\mu_J^2$. From \eqref{schema_2in1out} we have that, at the equilibrium, $z^1$ satisfies
$$
\frac{z_1}{z}G(z,\beta)=G(x,z).
$$
Moreover, 
$G(z,\beta)
\stackrel{\eqref{DS3}+\eqref{11i}}{=}
f(\alphauno)+f(\alphadue)$ 
and 
$G(x,z)
\stackrel{\eqref{DS1}+\eqref{11i}}{=}
f(\alphauno)$. 
Then, 
$$
z^1=\frac{f(\alphauno)}{f(\alphauno)+f(\alphadue)}\ z
\quad\textrm{ and }\quad z^2=z-z^1.
$$

\item CASE A.2

\begin{equation}\label{12i}
\left\{
\begin{array}{l}
G(\alphauno,x)=f(\alphauno) \\ [2mm]
G(\alphadue,y)=f(y)
\end{array}
\right. 
\end{equation}
We can also assume that 
\begin{equation}\label{12ii}
y\neq \alphadue \qquad \textrm{ and } \qquad y\neq\alphadue^\compl
\end{equation}
because if this does not hold true, we come back to CASE A.1. 
We have
\begin{equation*}
\eqref{assumption_uvzw}+\eqref{12i}+\eqref{12ii}
\Rightarrow \alphadue<y \textrm{ and }  f(y)<f(\alphadue)
\end{equation*}
and then
\begin{equation}\label{12iv}
y > \alphadue^\compl > \sigma.
\end{equation}
Finally we have
\begin{equation}\label{14vi}
\eqref{DS2}+\eqref{12i} \Rightarrow G(y,z)=f(y)
\end{equation}
which is impossible because of \eqref{12iv}. 
We conclude that CASE A.2 is not possible.

\item CASE A.3

\begin{equation}\label{13i}
\left\{
\begin{array}{l}
G(\alphauno,x)=f(x) \\ [2mm]
G(\alphadue,y)=f(\alphadue)
\end{array}
\right. 
\end{equation}
This case is not possible analogously to CASE A.2.

\item CASE A.4

\begin{equation}\label{14i}
\left\{
\begin{array}{l}
G(\alphauno,x)=f(x) \\ [2mm]
G(\alphadue,y)=f(y)
\end{array}
\right. 
\end{equation}
We can also assume that 
\begin{equation}\label{14ii}
x\neq \alphauno,\alphauno^\compl \qquad \textrm{ and } \qquad y\neq\alphadue,\alphadue^\compl
\end{equation}
because if this does not hold true, we come back to the previous cases.
By \eqref{assumption_uvzw}+\eqref{14i}+\eqref{14ii} we get
\begin{equation*}
(\alphauno<x \textrm{ and } f(x)<f(\alphauno)) \quad \text{ and } \quad (\alphadue<y \textrm{ and } f(y)<f(\alphadue))
\end{equation*}
and then
\begin{equation}\label{14FeM}
x>\alphauno^\compl>\sigma \qquad \text{and} \qquad y>\alphadue^\compl>\sigma.
\end{equation}
In particular we have
\begin{equation*}
\eqref{DS2}+\eqref{14i} \Rightarrow G(y,z)=f(y)
\end{equation*}
which is impossible because of \eqref{14FeM}. 
We conclude that CASE A.4 is not possible.
\end{itemize}

\vskip0.5cm

\item CASE (B) Omitted for brevity, same ideas apply. 


\vskip0.5cm

\item CASE (C) 

\begin{equation}\label{2i}
\boxed{
f(\alphauno) > \frac{f(\beta)}{2},
\quad 
f(\alphadue) > \frac{f(\beta)}{2}
}
\end{equation}

\begin{itemize}
\item CASE C.1
\begin{equation}\label{21i}
\left\{
\begin{array}{l}
G(\alphauno,x)=f(\alphauno) \\ [2mm]
G(\alphadue,y)=f(\alphadue)
\end{array}
\right. 
\end{equation}
The quantity $G(z,\beta)$ can be in principle equal to $f(z)$, $f(\beta)$, or $f(\sigma)$. But \eqref{assumption_uvzw} $\Rightarrow G(z,\beta)\neq f(\sigma)$. Moreover, if $G(z,\beta)=f(\beta)$, by \eqref{DS3} and \eqref{21i} we would have $f(\alphauno)+f(\alphadue)=f(\beta)$, which is in contradiction with \eqref{2i}. Then, we must have
\begin{equation}\label{21ii}
G(z,\beta)=f(z),
\end{equation}
which, with \eqref{assumption_uvzw}, implies
\begin{equation}\label{21iii}
z\leq\beta \textrm{ and } f(z)\leq f(\beta).
\end{equation}
Then, 
$$
f(z)
\stackrel{\eqref{DS3}+\eqref{21i}+\eqref{21ii}}{=}
f(\alphauno)+f(\alphadue)
\stackrel{\eqref{2i}}{>}
f(\beta)
$$
which contradicts \eqref{21iii}. We conclude that CASE C.1 is not possible.

\item CASE C.2. Omitted for brevity, same ideas apply. 

\item CASE C.3. Omitted for brevity, same ideas apply. 

\item CASE C.4

\begin{equation}\label{24i}
\left\{
\begin{array}{l}
G(\alphauno,x)=f(x) \\ [2mm]
G(\alphadue,y)=f(y)
\end{array}
\right. 
\end{equation}
As in CASE A.4, we can also assume that 
\begin{equation}\label{24ii}
x\neq \alphauno,\alphauno^\compl \qquad \textrm{ and } \qquad y\neq\alphadue,\alphadue^\compl
\end{equation}
and we get analogously
\begin{equation}\label{24iii}
x>\alphauno^\compl>\sigma \quad \textrm{ and } \quad y>\alphadue^\compl>\sigma.
\end{equation}

The quantity $G(z,\beta)$ can be in principle equal to $f(z)$, $f(\beta)$, or $f(\sigma)$. But \eqref{assumption_uvzw} $\Rightarrow G(z,\beta)\neq f(\sigma)$.
By contradiction, we can also prove that 
$G(z,\beta)\neq f(z)$. 
Indeed, if $G(z,\beta)=f(z)$ we would have
\begin{equation}\label{24iii_bis}
z<\beta^\compl<\sigma
\end{equation}
and
\begin{equation}\label{24iii_ter}
f(\sigma)
\stackrel{\eqref{24iii}+\eqref{24iii_bis}}{=}
G(x,z)
\stackrel{\eqref{DS1}}{=}
G(\alphauno,x)
\end{equation}
which contradicts \eqref{24i}.
As a consequence, we must have $G(z,\beta)=f(\beta)$, and then, by \eqref{DS3} and \eqref{24i},
\begin{equation}\label{24iv}
f(x)+f(y)=f(\beta).
\end{equation}
Moreover, we have that 
\begin{equation}\label{24vii}
z>\sigma,
\end{equation}
since, if $z\leq\sigma$, we get a contradiction as before in \eqref{24iii_ter}.
Finally, we have
\begin{equation}\label{24viii}
\begin{split}
f(x)
\stackrel{\eqref{24i}}{=} 
G(\alphauno,x)
\stackrel{\eqref{DS1}}{=} 
G(x,z)
\stackrel{\eqref{24iii}+\eqref{24vii}}{=} 
f(z)
\stackrel{\eqref{24iii}+\eqref{24vii}}{=}
G(y,z)
\stackrel{\eqref{DS2}}{=} \\
G(\alphadue,y)
\stackrel{\eqref{24i}}{=} 
f(y),
\end{split}
\end{equation}
and, by \eqref{24iv} and \eqref{24viii},
$$
f(x)=f(y)=f(z)=\frac{f(\beta)}{2}.
$$
This is indeed a stationary point for the system.

\end{itemize}
\end{itemize}
\end{proof}


\subsection{Proof of Theorem \ref{teo:stability}.}\label{app:A2} 
\begin{proof}
Consider again the discrete dynamical system \eqref{DS} and the auxiliary variables \eqref{def:xyz}.
\begin{itemize}
\item CASE (A)\\
In a neighbourhood of the stationary point \eqref{puntostazionarioA} the system has the form
\begin{equation}
\left\{
\begin{array}{l}
x^{n+1}=F_1(x^n,y^n,z^n):=x^n-\DtsuDx\big(f(x^n)-f(\alphauno)\big)\\ [2mm]
y^{n+1}=F_2(x^n,y^n,z^n):=y^n-\DtsuDx\big(f(y^n)-f(\alphadue)\big)\\ [2mm]
z^{n+1}=F_3(x^n,y^n,z^n):=z^n-\DtsuDx\big(f(z^n)-f(x^n)-f(y^n)\big).
\end{array}
\right. 
\end{equation}
The Jacobian matrix of $F=(F_1,F_2,F_3)$ computed at the stationary point $(x,y,z)$ is 
$$
J_F=
\left(
\begin{array}{ccc}
1-\DtsuDx f'(x) & 0               & 0 \\
0               & 1-\DtsuDx f'(y) & 0 \\
\DtsuDx f'(x)   & \DtsuDx f'(y)   & 1-\DtsuDx f'(z) 
\end{array}
\right).
$$
The three eigenvalues are $1-\DtsuDx f'(x)$, $1-\DtsuDx f'(y)$, and $1-\DtsuDx f'(z)$. Their absolute value is strictly less than 1 since $x,y,z<\sigma$ and \eqref{CFLstandard}. 
\item CASE (B)\\
In a neighbourhood of the stationary point \eqref{puntostazionarioB} the system has the form
\begin{equation}
\left\{
\begin{array}{l}
x^{n+1}=F_1(x^n,y^n,z^n):=x^n-\DtsuDx\big(f(z^n)-f(x^n)\big)\\ [2mm]
y^{n+1}=F_2(x^n,y^n,z^n):=y^n-\DtsuDx\big(f(y^n)-f(\alphadue)\big)\\ [2mm]
z^{n+1}=F_3(x^n,y^n,z^n):=z^n-\DtsuDx\big(f(\beta)-f(z^n)-f(y^n)\big).
\end{array}
\right. 
\end{equation}
The Jacobian matrix of $F=(F_1,F_2,F_3)$ computed at the stationary point $(x,y,z)$ is 
$$
J_F=
\left(
\begin{array}{ccc}
1+\DtsuDx f'(x) & 0               & -\DtsuDx f'(z) \\
0               & 1-\DtsuDx f'(y) & 0 \\
0               & \DtsuDx f'(y)   & 1+\DtsuDx f'(z) 
\end{array}
\right).
$$
The three eigenvalues are $1+\DtsuDx f'(x)$, $1-\DtsuDx f'(y)$, and $1+\DtsuDx f'(z)$. Their absolute value is strictly less than 1 since $y<\sigma$, $x,z>\sigma$ and \eqref{CFLstandard}. 
\item CASE (C)\\
In a neighbourhood of the stationary point \eqref{puntostazionarioD} the system has the form
\begin{equation}
\left\{
\begin{array}{l}
x^{n+1}=F_1(x^n,y^n,z^n):=x^n-\DtsuDx\big(f(z^n)-f(x^n)\big)\\ [2mm]
y^{n+1}=F_2(x^n,y^n,z^n):=y^n-\DtsuDx\big(f(z^n)-f(y^n)\big)\\ [2mm]
z^{n+1}=F_3(x^n,y^n,z^n):=z^n-\DtsuDx\big(f(\beta)-f(z^n)-f(z^n)\big).
\end{array}
\right. 
\end{equation}
The Jacobian matrix of $F=(F_1,F_2,F_3)$ computed at the stationary point $(x,y,z)$ is 
$$
J_F=
\left(
\begin{array}{ccc}
1+\DtsuDx f'(x) & 0               & -\DtsuDx f'(z) \\
0               & 1+\DtsuDx f'(y) & -\DtsuDx f'(z) \\
0               & 0               & 1+2\DtsuDx f'(z) 
\end{array}
\right).
$$
The three eigenvalues are $1+\DtsuDx f'(x)$, $1+\DtsuDx f'(y)$, and $1+2\DtsuDx f'(z)$. Their absolute value is strictly less than 1 since $x,y,z>\sigma$ and \eqref{CFLstandard}. 
\end{itemize}
\end{proof}


\subsection{Proof of Theorem \ref{teo:RPteorico}.}\label{app:A3}

\begin{proof} 
The proof is constructive. 
We refer to Appendix \ref{sec:Riemann_and_half-Riemann_intro} for some basic notions about (half) Riemann problem. 
We have to solve a left-half Riemann problem for the two equations in $u$ and $v$, a left- and a right-half Riemann problem (with different fluxes) for the equation in $z$, and a right-half Riemann problem for the equation in $w$, 
under the constraints \eqref{boundary_0} and \eqref{boundary_Dx}, initial conditions satisfying \eqref{assumption_uvzw} and \eqref{zc=sigma}, and with the choice of fluxes \eqref{choice_at_junction_in} and \eqref{choice_at_junction_out}.
In the proof we employ the notation $\compl$ defined by \eqref{tau}.

We recall here the definition of the flux $\hc$ given in \eqref{def:hc_primavolta},
\begin{equation}\label{def:hc}
\hc(\cdot):=f(\cdot)+f(\sigma)
\end{equation}
and we look for the fluxes $\hat h(\cdot)$, $\tilde h(\cdot)$ and the constants $\tilde u, \tilde v, \hat z, \tilde z, \hat w$ (see Fig.\ \ref{fig:RiemannPb}) such that $u(x$=$0-,\ t$=$0+)=\tilde u$, $v(0-,0+)=\tilde v$, $z(0+,0+)=\hat z$, $z(\Dx-,0+)=\tilde z$, $w(\Dx+,0+)=\hat w$ with
\begin{equation}
(\tilde u=\ul \text{ or } \tilde u\in N(\ul)) \quad\mbox{and}\quad f(\tilde u) = \min\{ G_f(\tilde u,\sigma), G_f(\sigma,\hat z)\} \label{A_def_utilde}
\end{equation}
\begin{equation}
(\tilde v=\vl \text{ or } \tilde v\in N(\vl)) \quad\mbox{and}\quad f(\tilde v) = \min\{ G_f(\tilde v,\sigma), G_f(\sigma,\hat z)\} \label{A_def_vtilde}
\end{equation}
\begin{equation}
\big((\hat h,\hat z)=(\hc,\zc) \text{ or } (\hat h,\hat z)\in \mathcal P(\hc,\zc)\big) \quad\mbox{and}\quad \hat h(\hat z) = f(\tilde u) + f(\tilde v) \label{A_def_zhat}
\end{equation}
\begin{equation}
\begin{split}
\big((\tilde h,\tilde z)=(\hc,\zc) \text{ or } (\tilde h,\tilde z)\in \mathcal N(\hc,\zc)\big) \quad\mbox{and} \quad \tilde h(\tilde z) =\\ \min\{G_{\tilde h}(\tilde z,\sigma),  G_f(\sigma,\hat w)\}
\end{split} 
\label{A_def_ztilde}
\end{equation}
\begin{equation}
(\hat w=\wr \text{ or } \hat w\in P(w_r)) \quad\mbox{and}\quad f(\hat w)=\tilde h(\tilde z)
\label{A_def_what}
\end{equation}
where the sets $N$, $P$, $\mathcal N$, $\mathcal P$ are defined in Definitions \ref{def:N}, \ref{def:P}, \ref{def:Ngen}, \ref{def:Pgen}, respectively.
\begin{figure}[h!]
\begin{center}
\begin{psfrags}
\psfrag{x}{$x$} \psfrag{t}{$t$}
\psfrag{Us}{$\ul$} \psfrag{Vs}{$\vl$}
\psfrag{Ut}{$\tilde u$} \psfrag{Vt}{$\tilde v$} 
\psfrag{Rc}{$(\hat h,\hat z)$} \psfrag{Rd}{\hskip-11pt$(\hc,\zc)$} 
\psfrag{FLUSSI}{\hskip-7pt$f(\tilde u)+f(\tilde v) = \hat h(\hat z)$} 
\psfrag{FLUSSI2}{\hskip-11pt$\tilde h(\tilde z)=f(\hat w)$} 
\psfrag{Re}{$(\tilde h,\tilde z)$} 
\psfrag{Wp}{$w_r$} \psfrag{Wm}{$\hat w$} 
\psfrag{Dx}{\hskip-4pt$\Dx$}\psfrag{0}{$0$}
\psfrag{Lt}{\hskip-8pt$L_1^u$,$L_1^v$}
\psfrag{Lh}{$L_4$}
\psfrag{L1}{$\!\!\!L_2$}\psfrag{L2}{\hskip-5pt$L_3$}
\includegraphics[width=0.8\textwidth]{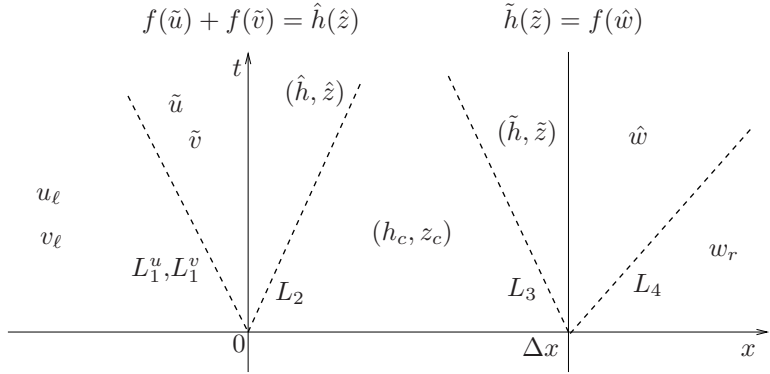}
\end{psfrags}
\end{center}
\caption{Solution of the Riemann problem in the plane $(x,t)$ just after the initial time.
}
\label{fig:RiemannPb}
\end{figure}
%


\begin{itemize}
\item CASE (A) 
\begin{equation}\label{(A)Riemann_proof}
\boxed{f(\ul)+f(\vl) < f(\wr)}
\end{equation}
We claim that the solution is:
\begin{equation}\label{A_prima_soluz}
\boxed{
\begin{array}{c}
\tilde u=\ul, \quad 
\tilde v=\vl, \quad 
\hat z<\sigma \textrm{ s.t. } f(\hat z)=f(\ul)+f(\vl), 
\\ [1mm]
\tilde z=\wr, \quad 
\hat w=\wr, \quad 
\hat h=f, \quad 
\tilde h=f.
\end{array}
}
\end{equation}
It is important to note that such a $\hat z$ actually exists since
$$
f(\ul)+f(\vl)
\stackrel{\eqref{(A)Riemann_proof}}{<}
f(\wr)\leq f(\sigma).
$$
Let us now verify that the fluxes/constants in \eqref{A_prima_soluz} are actually solution of \eqref{A_def_utilde}--\eqref{A_def_what}.
\begin{itemize}
\item[\boxed{\ref{A_def_utilde}}] We have $\tilde u=\ul$ and 
\begin{equation*}
\begin{split}
f(\tilde u)=f(\ul)=
\min\{f(\ul),f(\sigma)\}=
\min\{G_f(\ul,\sigma),G_f(\sigma,\hat z)\}= \\
\min\{G_f(\tilde u,\sigma),G_f(\sigma,\hat z)\}.
\end{split}
\end{equation*}
\item[\boxed{\ref{A_def_vtilde}}] $\tilde v$ is analogous.
\item[\boxed{\ref{A_def_zhat}}] We prove that $(\hat h,\hat z)\in \mathcal P(\hc,\zc)$. To this end, we show that the wave $L_2:=\big((\hat h,\hat z),(\hc,\zc)\big)$ is a strictly positive shock, see Fig.\ \ref{fig:RiemannPb_A}. We have
$$
(\hat z<\sigma=\zc) \Rightarrow (f'(\hat z)>f'(\zc)) 
\stackrel{\eqref{def:hc}+\eqref{A_prima_soluz}}{\Rightarrow}
(\hat h'(\hat z)>\hc'(\zc))
$$ 
and then, referring to Appendix \ref{sec:RP_gen_intro}, $L_2$ is a shock. Moreover, 
$$
\hc(\zc)
\stackrel{\eqref{def:hc}}{=}
f(\zc)+f(\sigma)
\stackrel{\eqref{zc=sigma}}{=}
2f(\sigma)>f(\ul)+f(\vl)
\stackrel{\eqref{A_prima_soluz}}{=}
f(\hat z)
\stackrel{\eqref{A_prima_soluz}}{=}
\hat h(\hat z)
$$
and then, by \eqref{RHgen} and $\zc>\hat z$, $L_2$ is strictly positive.
Finally, the equality $\hat h(\hat z) = f(\tilde u) + f(\tilde v)$ is trivially verified.
\item[\boxed{\ref{A_def_ztilde}}] 
We prove that $(\tilde h,\tilde z)\in \mathcal N(\hc,\zc)$. To this end, we show that $L_3:=\big((\hc,\zc),(\tilde h,\tilde z)\big)$ is a strictly negative shock, see Fig.\ \ref{fig:RiemannPb_A}. We have
$$
(\tilde z
\stackrel{\eqref{A_prima_soluz}}{=}
\wr
>\sigma=\zc) 
\Rightarrow (f'(\zc)>f'(\tilde z)) 
\stackrel{\eqref{def:hc}+\eqref{A_prima_soluz}}{\Rightarrow}
(\hc'(\zc)>\tilde h'(\tilde z))
$$ 
and then, referring to Appendix\ \ref{sec:RP_gen_intro}, $L_3$ is a shock. Moreover, 
$$
\tilde h(\tilde z)
\stackrel{\eqref{A_prima_soluz}}{=}
f(\tilde z)<2f(\sigma)
\stackrel{\eqref{zc=sigma}}{=}
f(\zc)+f(\sigma)
\stackrel{\eqref{def:hc}}{=}
\hc(\zc)
$$
and then, by \eqref{RHgen} and $\tilde z>\zc$, $L_3$ is strictly negative.
Finally, we have
\begin{equation*}
\begin{split}
\tilde h(\tilde z)
\stackrel{\eqref{A_prima_soluz}}{=}
f(\tilde z)
\stackrel{\eqref{A_prima_soluz}}{=}
f(\wr)=
\min\{f(\sigma),f(\wr)\}=\\
\min\{G_f(\tilde z,\sigma),G_f(\sigma,\wr)\}
\stackrel{\eqref{A_prima_soluz}}{=}
\min\{G_{\tilde h}(\tilde z,\sigma),G_f(\sigma,\hat w)\}.
\end{split}
\end{equation*}
\item[\boxed{\ref{A_def_what}}] We have $\hat w=\wr$ and 
$f(\hat w)
\stackrel{\eqref{A_prima_soluz}}{=}
f(\wr)
\stackrel{\eqref{A_prima_soluz}}{=}
\tilde h(\tilde z)$.
\end{itemize}

\medskip

Summarizing, we have that the waves $L_1^u=(\ul,\tilde u)$, $L_1^v=(\vl,\tilde v)$, and $L_4=(\hat w,\wr)$ in Fig.\ \ref{fig:RiemannPb} are not even created, while
\begin{figure}[h!]
\begin{center}
\begin{psfrags}
\psfrag{x}{$x$} \psfrag{t}{$t$}
\psfrag{Us}{$\ul=\tilde u$} \psfrag{Vs}{$\vl=\tilde v$}
\psfrag{What2}{$\hat w_2$} \psfrag{Vt}{$\tilde v$} 
\psfrag{Rc}{$(\hat h,\hat z)$} \psfrag{Rd}{$\!\!\!\!\!\!(\hc,\zc)$} 
\psfrag{Re}{$\!\!\!\!(\tilde h,\tilde z)$}
\psfrag{Wp}{$w_r=\hat w$} 
\psfrag{Dx}{\hskip-4pt$\Dx$}\psfrag{0}{$0$}
\psfrag{L3}{\hskip-4pt$L_6$}
\psfrag{Lt}{\hskip-4pt$L_1^u$}\psfrag{Lv}{\hskip-4pt$L_1^v$}
\psfrag{Lhat}{$L_4$}
\psfrag{Lhat2}{\hskip-4pt$L_7$}
\psfrag{L1}{$L_2$}\psfrag{L2}{$L_3$}
\psfrag{Lb}{$L_5$}
\includegraphics[width=0.8\textwidth]{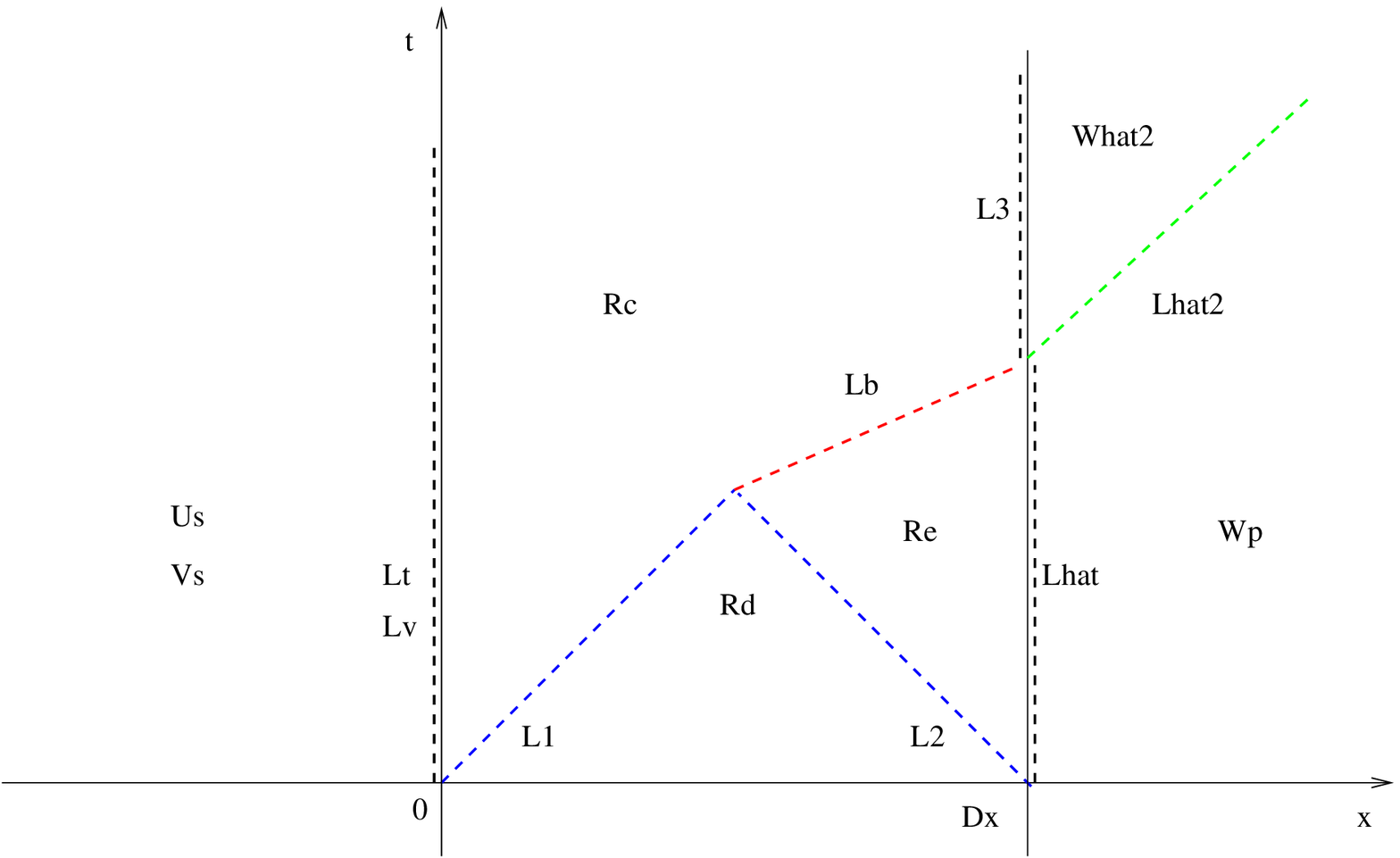}
\end{psfrags}
\end{center}
\caption{CASE (A): Solution of the Riemann problem. }
\label{fig:RiemannPb_A}
\end{figure}
the wave $L_2$ is a strictly positive shock and the wave $L_3$ is a strictly negative shock. When $L_2$ meets $L_3$, a new wave $L_5:=\big((\hat h,\hat z),(\tilde h,\tilde z)\big)$ arises, see Fig.\ \ref{fig:RiemannPb_A}. Let us prove that $L_5$ is a strictly positive shock. We have 
$$
(\hat z<\sigma<\tilde z) \Rightarrow (f'(\hat z)>f'(\tilde z)) 
\stackrel{\eqref{A_prima_soluz}}{\Rightarrow}
(\hat h'(\hat z)>\tilde h'(\tilde z))
$$ 
and then, referring to  Appendix\ \ref{sec:RP_gen_intro}, $L_5$ is a shock. Moreover,
$$
\tilde h(\tilde z)
\stackrel{\eqref{A_prima_soluz}}{=}
f(\tilde z)
\stackrel{\eqref{A_prima_soluz}}{=}
f(\wr)
\stackrel{\eqref{(A)Riemann_proof}}{>}
f(\ul)+f(\vl)
\stackrel{\eqref{A_prima_soluz}}{=}
f(\hat z)
\stackrel{\eqref{A_prima_soluz}}{=}
\hat h(\hat z)
$$
and then, by \eqref{RHgen} and $\tilde z>\hat z$, $L_5$ is strictly positive.
\medskip

When $L_5$ reaches the line $\{x=\Dx\}$ two new waves start. In the same spirit of \eqref{A_def_ztilde}--\eqref{A_def_what} we define the new flux $\tilde h_2$ and the new constants $\tilde z_2$ and $\hat w_2$ such that 
\begin{equation}
\begin{split}
\big((\tilde h_2,\tilde z_2)=(\hat h,\hat z) \text{ or } (\tilde h_2,\tilde z_2)\in \mathcal N(\hat h,\hat z)\big) \quad\mbox{and}\quad \tilde h_2(\tilde z_2) =\\ \min\{G_{\tilde h_2}(\tilde z_2,\sigma), G_f(\sigma,\hat w_2)\}
\end{split} 
\label{A_def_ztilde2}
\end{equation}
\begin{equation}
(\hat w_2=\wr \text{ or } \hat w_2\in P(w_r)) \quad\mbox{and}\quad f(\hat w_2)=\tilde h_2(\tilde z_2).
\label{A_def_what2}
\end{equation}
We claim that the solution is
\begin{equation}\label{A_seconda_soluz}
\boxed{
\tilde z_2=\hat z, \quad 
\hat w_2=\hat z, \quad 
\tilde h_2=\hat h=f.
}
\end{equation}
Let us verify that flux/constants in \eqref{A_seconda_soluz} are actually solution of \eqref{A_def_ztilde2}--\eqref{A_def_what2}.
\begin{itemize}
\item[\boxed{\ref{A_def_ztilde2}}]  We have 
$(\tilde h_2,\tilde z_2)=(\hat h,\hat z)$
and 
\begin{equation*}
\begin{split}
\tilde h_2(\tilde z_2)=
f(\hat z)=
\min\{f(\hat z),f(\sigma)\}=
\min\{G_f(\hat z,\sigma),G_f(\sigma,\hat z)\}=\\
\min\{G_{\tilde h_2}(\tilde z_2,\sigma),G_f(\sigma,\hat w_2)\}.
\end{split}
\end{equation*}
\item[\boxed{\ref{A_def_what2}}] We prove that $\hat w_2=\hat z \in P(\wr)$. By Proposition \ref{prop:P}, this is true if $\hat z<\sigma$ and $f(\hat z)\leq f(\wr)$, which, in turn, comes from \eqref{(A)Riemann_proof} and \eqref{A_prima_soluz}.
Finally, by \eqref{A_seconda_soluz} we have
$$
f(\hat w_2)=
f(\hat z)=
\tilde h_2(\tilde z_2).
$$
\end{itemize} 

\medskip

Summarizing, we have that the wave $L_6:=\big((\hat h,\hat z),(\tilde h_2,\tilde z_2)\big)$ is not created, see Fig.\ \ref{fig:RiemannPb_A}. Moreover, since $\hat w_2=\hat z<\sigma<\wr$ the wave $\hat L_7:=(\hat w_2,\wr)$ is a shock. It is strictly positive again by \eqref{(A)Riemann_proof} and \eqref{A_prima_soluz}.

The proof is concluded by choosing 
$$
\bar u=\tilde u, \quad \bar v=\tilde v, \quad \bar z=\hat z, \quad \bar w=\hat w_2.
$$

\item CASE (B) 
\begin{equation}\label{(B)Riemann_proof}
\boxed{f(\ul) + f(\vl) > f(\wr),\quad f(\vl) < \frac{f(\wr)}{2}}
\end{equation}
Let us start again from the initial time $t=0$ and find the fluxes $\hat h(\cdot)$, $\tilde h(\cdot)$ and the constants $\tilde u, \tilde v, \hat z, \tilde z, \hat w$ such that \eqref{A_def_utilde}--\eqref{A_def_what} hold true.
We claim that the solution is
\begin{equation}\label{B_prima_soluz}
\boxed{
\begin{array}{c}
\tilde u=\ul, \quad 
\tilde v=\vl, \quad 
\hat z=\ul, \quad
\tilde z>\sigma \textrm{ s.t. } f(\tilde z)=f(\wr)-f(\vl), \quad \\ [1mm]
\hat w=\wr, \quad
\hat h(\cdot)=\tilde h(\cdot)=f(\cdot)+f(\vl).
\end{array}
}
\end{equation}
Note that $f(\wr)-f(\vl)>0$ by \eqref{(B)Riemann_proof} and then such a $\tilde z$ exists.

Let us verify that the fluxes/constants in \eqref{B_prima_soluz} are actually solution of \eqref{A_def_utilde}--\eqref{A_def_what}.
\begin{itemize}
\item[\boxed{\ref{A_def_utilde}}] As in CASE (A).
\item[\boxed{\ref{A_def_vtilde}}] As in CASE (A).
\item[\boxed{\ref{A_def_zhat}}] We prove that $(\hat h,\hat z)\in \mathcal P(\hc,\zc)$. To this end, we show that the wave $L_2$ is a strictly positive shock, see Fig.\ \ref{fig:RiemannPb_B}. We have
$$
(\hat z<\sigma=\zc) \Rightarrow (f'(\hat z)>f'(\zc)) 
\stackrel{\eqref{def:hc}+\eqref{B_prima_soluz}}{\Rightarrow}
(\hat h'(\hat z)>\hc'(\zc))
$$ 
and then, referring to Appendix\ \ref{sec:RP_gen_intro}, $L_2$ is a shock. Moreover, 
$$
\hc(\zc)
\stackrel{\eqref{def:hc}}{=}
f(\zc)+f(\sigma)
\stackrel{\eqref{zc=sigma}}{=}
2f(\sigma)
>
f(\ul)+f(\vl)
\stackrel{\eqref{B_prima_soluz}}{=}
f(\hat z)+f(\vl)
\stackrel{\eqref{B_prima_soluz}}{=}
\hat h(\hat z)
$$
and then, by \eqref{RHgen} and $\zc>\hat z$, $L_2$ is strictly positive.
Finally, the equality $\hat h(\hat z) = f(\tilde u) + f(\tilde v)$ is trivially verified.
\item[\boxed{\ref{A_def_ztilde}}] 
We prove that $(\tilde h,\tilde z)\in \mathcal N(\hc,\zc)$. To this end, we show that $L_3$ is a strictly negative shock, see Fig.\ \ref{fig:RiemannPb_B}. We have
$$
(\tilde z>\sigma=\zc) \Rightarrow (f'(\zc)>f'(\tilde z)) 
\stackrel{\eqref{def:hc}+\eqref{B_prima_soluz}}{\Rightarrow}
(\hc'(\zc)>\tilde h'(\tilde z))
$$ 
and then, referring to Appendix\ \ref{sec:RP_gen_intro}, $L_3$ is a shock. Moreover, 
$$
\tilde h(\tilde z)
\stackrel{\eqref{B_prima_soluz}}{=}
f(\tilde z)+f(\vl)
<
2f(\sigma)
\stackrel{\eqref{zc=sigma}}{=}
f(\zc)+f(\sigma)
\stackrel{\eqref{def:hc}}{=}
\hc(\zc)
$$
and then, by \eqref{RHgen} and $\tilde z>\zc$, $L_3$ is strictly negative.
Finally, we have
\begin{equation*}
\tilde h(\tilde z)
\stackrel{\eqref{B_prima_soluz}}{=}
f(\wr)=
\min\{f(\sigma)+f(\vl),f(\wr)\}
\stackrel{\eqref{B_prima_soluz}}{=}
\min\{G_{\tilde h}(\tilde z,\sigma),G_f(\sigma,\hat w)\}.
\end{equation*}
\item[\boxed{\ref{A_def_what}}] We have $\hat w=\wr$ and 
$f(\hat w)
\stackrel{\eqref{B_prima_soluz}}{=}
f(\wr)
\stackrel{\eqref{B_prima_soluz}}{=}
\tilde h(\tilde z)$.
\end{itemize}

Summarizing, the situation is similar to the case (A), see Fig.\ \ref{fig:RiemannPb_B}: $L_1^u$, $L_1^v$, and $L_4$ are not created. $L_2$ is a strictly positive shock and $L_3$ is a strictly negative shock. When $L_2$ meets $L_3$, a new wave $L_5=\big((\hat h,\hat z),(\tilde h,\tilde z)\big)$ arises. Let us prove that $L_5$ is a strictly negative shock.
\begin{figure}[h!]
\begin{center}
\begin{psfrags}
\psfrag{x}{$x$} \psfrag{t}{$t$}
\psfrag{Us}{$\ul=\tilde u$} \psfrag{Vs}{$\vl=\tilde v$}
\psfrag{Ut2}{$\tilde u_2$} \psfrag{Vt2}{$\tilde v_2=\vl$} 
\psfrag{Zhat2}{$\hat z_2$}
\psfrag{What2}{$\hat w_2$}
\psfrag{Rc}{$(\hat h,\hat z)$} \psfrag{Rd}{$\!\!\!\!\!\!(\hc,\zc)$} 
\psfrag{FLUSSI}{$f(\tilde u)+f(\tilde v) = h(\hat z)$} 
\psfrag{FLUSSI2}{$h(\tilde z)=f(\hat w)$} \psfrag{Re}{$(\tilde h,\tilde z)$} 
\psfrag{Wp}{$w_r=\hat w$} 
\psfrag{Dx}{\hskip-4pt$\Dx$}\psfrag{0}{$0$}
\psfrag{Lt}{\hskip-5pt$L_1^u$}\psfrag{Lv}{\hskip-5pt$L_1^v$}\psfrag{Lhat}{$L_4$}
\psfrag{Lt2}{\hskip-3pt$L_7^u$}
\psfrag{Ltv}{\hskip-2pt$L_7^v$}\psfrag{Lhat2}{$\hat L_2$}
\psfrag{L1}{$L_2$}\psfrag{L2}{$L_3$}
\psfrag{Lb}{$L_5$}\psfrag{L3}{$L_6$}\psfrag{L4}{$L_6$}
\includegraphics[width=0.8\textwidth]{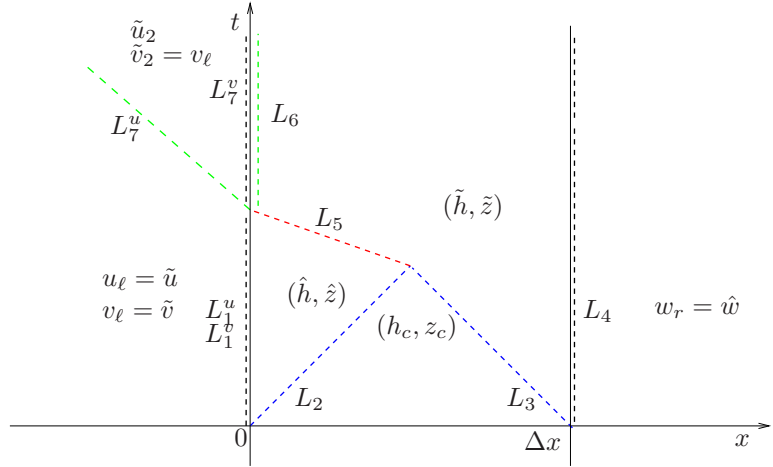}
\end{psfrags}
\end{center}
\caption{CASE (B): Solution of the Riemann problem. }
\label{fig:RiemannPb_B}
\end{figure}

We have 
$$
(\hat z<\sigma<\tilde z) \Rightarrow (f'(\hat z)>f'(\tilde z)) 
\stackrel{\eqref{def:hc}+\eqref{B_prima_soluz}}{\Rightarrow}
(\hat h'(\hat z)>\tilde h'(\tilde z))
$$ 
and then, referring to Appendix\ \ref{sec:RP_gen_intro}, $L_5$ is a shock. Moreover,
$$
\tilde h(\tilde z)
\stackrel{\eqref{B_prima_soluz}}{=}
f(\tilde z)+f(\vl)
\stackrel{\eqref{B_prima_soluz}}{=}
f(\wr)
\stackrel{\eqref{(B)Riemann_proof}}{<}
f(\ul)+f(\vl)
\stackrel{\eqref{B_prima_soluz}}{=}
f(\hat z)+f(\vl)
\stackrel{\eqref{B_prima_soluz}}{=}
\hat h(\hat z)
$$
and then, by \eqref{RHgen} and $\tilde z>\hat z$, $L_5$ is strictly negative.

\medskip

When $L_5$ reaches the line $\{x=0\}$, three new waves $L_6:=\big((\hat h_2,\hat z_2),(\tilde h,\tilde z)\big)$, $L_7^u:=(\tilde u,\tilde u_2)$, and $L_7^v:=(\tilde v,\tilde v_2)$ start, see Fig.\ \ref{fig:RiemannPb_B}. Then we look for the flux $\hat h_2$ and the constants $\tilde u_2, \tilde v_2, \hat z_2$ such that
\begin{eqnarray}
& &(\tilde u_2=\tilde u \text{ or } \tilde u_2\in N(\tilde u)) \quad\mbox{and}\quad f(\tilde u_2) = \min\{G_f(\tilde u_2,\sigma),G_f(\sigma,\hat z_2)\} \label{B_def_utilde2}
\\ [1mm]
& &(\tilde v_2=\tilde v \text{ or } \tilde v_2\in N(\tilde v)) \quad\mbox{and}\quad f(\tilde v_2) = \min\{G_f(\tilde v_2,\sigma),G_f(\sigma,\hat z_2)\} \label{B_def_vtilde2}
\\ [1mm]
& &((\hat h_2,\hat z_2)=(\tilde h,\tilde z) \text{ or } (\hat h_2,\hat z_2)\in \mathcal P(\tilde h,\tilde z)) \quad\mbox{and}\quad \hat h_2(\hat z_2)=f(\tilde u_2)+f(\tilde v_2).
\label{B_def_zhat2}
\end{eqnarray}
We claim that the solution is
\begin{equation}\label{B_seconda_soluz}
\boxed{\tilde u_2>\sigma \textrm{ s.t. } f(\tilde u_2)=f(\wr)-f(\vl), \quad 
\tilde v_2=\vl, \quad 
\hat z_2=\tilde u_2, \quad
\hat h_2=\tilde h.}
\end{equation}
Let us verify that flux/constants in \eqref{B_seconda_soluz} are actually solution of \eqref{B_def_utilde2}--\eqref{B_def_zhat2}.
\begin{itemize}
\item[\boxed{\ref{B_def_utilde2}}]  We prove that $\tilde u_2 \in N(\tilde u=\ul)$. By Proposition \ref{prop:N}, this is true if $\tilde u_2\geq\sigma$ and $f(\tilde u_2)\leq f(\ul)$. We have indeed $\tilde u_2>\sigma$ and
$$
f(\tilde u_2)
\stackrel{\eqref{B_seconda_soluz}}{=}
f(\wr)-f(\vl)
\stackrel{\eqref{(B)Riemann_proof}}{<}
f(\ul).
$$
Moreover,
$$
f(\tilde u_2)
\stackrel{\eqref{B_seconda_soluz}}{=}
f(\wr)-f(\vl)=
\min\{f(\sigma),f(\wr)-f(\vl)\}
\stackrel{\eqref{B_seconda_soluz}}{=}
\min\{G_f(\tilde u_2,\sigma),G_f(\sigma,\hat z_2)\}.
$$
\item[\boxed{\ref{B_def_vtilde2}}] We have $\tilde v_2=\vl=\tilde v$.
Moreover,
$$
f(\tilde v_2)
\stackrel{\eqref{B_seconda_soluz}}{=}
f(\vl)
\stackrel{\eqref{(B)Riemann_proof}}{=}
\min\{f(\vl),f(\wr)-f(\vl)\}
\stackrel{\eqref{B_seconda_soluz}}{=}
\min\{G_f(\tilde v_2,\sigma),G_f(\sigma,\hat z_2)\}.
$$
\item[\boxed{\ref{B_def_zhat2}}] We have $(\hat h_2,\hat z_2)=(\tilde h,\tilde z)$.
Moreover,
$$
\hat h_2(\hat z_2)
\stackrel{\eqref{B_seconda_soluz}}{=}
\tilde h(\tilde z)
\stackrel{\eqref{B_prima_soluz}}{=}
f(\wr)-f(\vl)+f(\vl)
=
f(\wr)
\stackrel{\eqref{B_seconda_soluz}}{=}
f(\tilde u_2)+f(\tilde v_2).
$$
\end{itemize}
Summarizing, we have that the wave $L_6$ is not created. 
Let us study the waves $L_7^u$ and $L_7^v$.
$L_7^u$ is a strictly negative shock since 
$\tilde u
\stackrel{\eqref{B_prima_soluz}}{=}
\ul
\stackrel{\eqref{assumption_uvzw}}{<}
\sigma
\stackrel{\eqref{B_seconda_soluz}}{<}
\tilde u_2$ 
and 
$$
f(\tilde u)
\stackrel{\eqref{B_prima_soluz}}{=}
f(\ul)
\stackrel{\eqref{(B)Riemann_proof}}{>}
f(\wr)-f(\vl)
\stackrel{\eqref{B_seconda_soluz}}{=}
f(\tilde u_2).
$$
$L_7^v$ instead is not created since $(\tilde v,\tilde v_2)=(\vl,\vl)$.

The proof is concluded by choosing 
$$
\bar u=\tilde u_2, \quad \bar v=\tilde v_2, \quad \bar z=\tilde z, \quad \bar w=\hat w.
$$

\item CASE (C)
\begin{equation}\label{(C)Riemann_proof}
\boxed{f(\ul) > \frac{f(\wr)}{2},\quad f(\vl) > \frac{f(\wr)}{2}}
\end{equation}
Let us start again from the initial time $t=0$ and find the fluxes $\hat h(\cdot)$, $\tilde h(\cdot)$ and the constants $\tilde u, \tilde v, \hat z, \tilde z, \hat w$ such that \eqref{A_def_utilde}--\eqref{A_def_what} hold true.
We claim that the solution is
\begin{equation}\label{C_prima_soluz}
\boxed{
\begin{array}{c}
\tilde u=\ul, \quad 
\tilde v=\vl, \quad 
\hat z=\max\{\ul,\vl\}, \quad 
\tilde z>\sigma \textrm{ s.t. } f(\tilde z)=\frac{f(\wr)}{2}, \quad
\\ [2mm]
\hat w=\wr, \quad
\hat h(\cdot)=f(\cdot)+\min\{f(\ul),f(\vl)\}, \quad
\tilde h(\cdot)=f(\cdot)+\frac{f(\wr)}{2}.
\end{array}
}
\end{equation}
We prove only that \eqref{A_def_zhat} and \eqref{A_def_ztilde} hold true. 
\begin{itemize}
\item[\boxed{\ref{A_def_zhat}}] We prove that $(\hat h,\hat z)\in \mathcal P(\hc,\zc)$. To this end, we show that the wave $L_2$ is a strictly positive shock, see Fig.\ \ref{fig:RiemannPb_C}. We have
$$
(\hat z=\max\{\ul,\vl\}<\sigma=\zc) \Rightarrow (f'(\hat z)>f'(\zc)) 
\stackrel{\eqref{def:hc}+\eqref{C_prima_soluz}}{\Rightarrow}
(\hat h'(\hat z)>\hc'(\zc))
$$ 
and then, referring to Appendix\ \ref{sec:RP_gen_intro}, $L_2$ is a shock. Moreover, 
\begin{equation*}
\begin{split}
\hc(\zc)
\stackrel{\eqref{def:hc}}{=}
f(\zc)+f(\sigma)
\stackrel{\eqref{zc=sigma}}{=}
2f(\sigma)
>
f(\ul)+f(\vl)
\stackrel{\eqref{assumption_uvzw}}{=} \\
f(\max\{\ul,\vl\})+\min\{f(\ul),f(\vl)\}
\stackrel{\eqref{C_prima_soluz}}{=}
\hat h(\hat z)
\end{split}
\end{equation*}
and then, by \eqref{RHgen} and $\zc>\hat z$, $L_2$ is strictly positive.
Finally, the equality $\hat h(\hat z) = f(\tilde u) + f(\tilde v)$ is trivially verified.
\item[\boxed{\ref{A_def_ztilde}}] 
We prove that $(\tilde h,\tilde z)\in \mathcal N(\hc,\zc)$. To this end, we show that $L_3$ is a strictly negative shock, see Fig.\ \ref{fig:RiemannPb_C}. We have
$$
(\tilde z>\sigma=\zc) \Rightarrow (f'(\zc)>f'(\tilde z)) 
\stackrel{\eqref{def:hc}+\eqref{C_prima_soluz}}{\Rightarrow}
(\hc'(\zc)>\tilde h'(\tilde z))
$$ 
and then, referring to Appendix\ \ref{sec:RP_gen_intro}, $L_3$ is a shock. Moreover, 
$$
\tilde h(\tilde z)
\stackrel{\eqref{C_prima_soluz}}{=}
f(\wr)
<
2f(\sigma)
\stackrel{\eqref{zc=sigma}}{=}
f(\zc)+f(\sigma)
\stackrel{\eqref{def:hc}}{=}
\hc(\zc)
$$
and then, by \eqref{RHgen} and $\tilde z>\zc$, $L_3$ is strictly negative.
Finally, we have
\begin{equation*}
\tilde h(\tilde z)
\stackrel{\eqref{C_prima_soluz}}{=}
f(\wr)=
\min\left\{f(\sigma)+\frac{f(\wr)}{2},f(\wr)\right\}
\stackrel{\eqref{C_prima_soluz}}{=}
\min\{G_{\tilde h}(\tilde z,\sigma),G_f(\sigma,\hat w)\}.
\end{equation*}
\end{itemize}

Summarizing, the situation is similar to the case (A) and (B), see Fig.\ \ref{fig:RiemannPb_C}: $L_1^u$, $L_2^v$, and $L_4$ are not created. $L_2$ is a strictly positive shock and $L_3$ is a strictly negative shock. When $L_2$ meets $L_3$, a new wave $L_5=\big((\hat h,\hat z),(\tilde h,\tilde z)\big)$ arises. Let us prove that $L_5$ is a strictly negative shock.
\begin{figure}[h!]
\begin{center}
\begin{psfrags}
\psfrag{x}{\hskip5pt$x$} \psfrag{t}{$t$}
\psfrag{Us}{$\ul=\tilde u$} \psfrag{Vs}{$\vl=\tilde v$}
\psfrag{Rc}{$(\hat h,\hat z)$} \psfrag{Rd}{$\!\!\!\!\!\!(\hc,\zc)$} 
\psfrag{Re}{$(\tilde h,\tilde z)$}
\psfrag{Dx}{\hskip-3pt$\Dx$}\psfrag{0}{$0$}
\psfrag{Vt2}{$\tilde u_2,\tilde v_2$}
\psfrag{Lhat}{$L_4$}\psfrag{Wp}{$w_r = \hat w$}
\psfrag{Lt}{\hskip-5pt$L_1^u$}\psfrag{Lv}{\hskip-5pt$L_1^v$}
\psfrag{Lt2}{\hskip-10pt$L^u_7$, $L^v_7$}
\psfrag{L1}{$L_2$}\psfrag{L2}{$L_3$}
\psfrag{Lb}{$L_5$}\psfrag{L3}{$L_6$}
\includegraphics[width=0.8\textwidth]{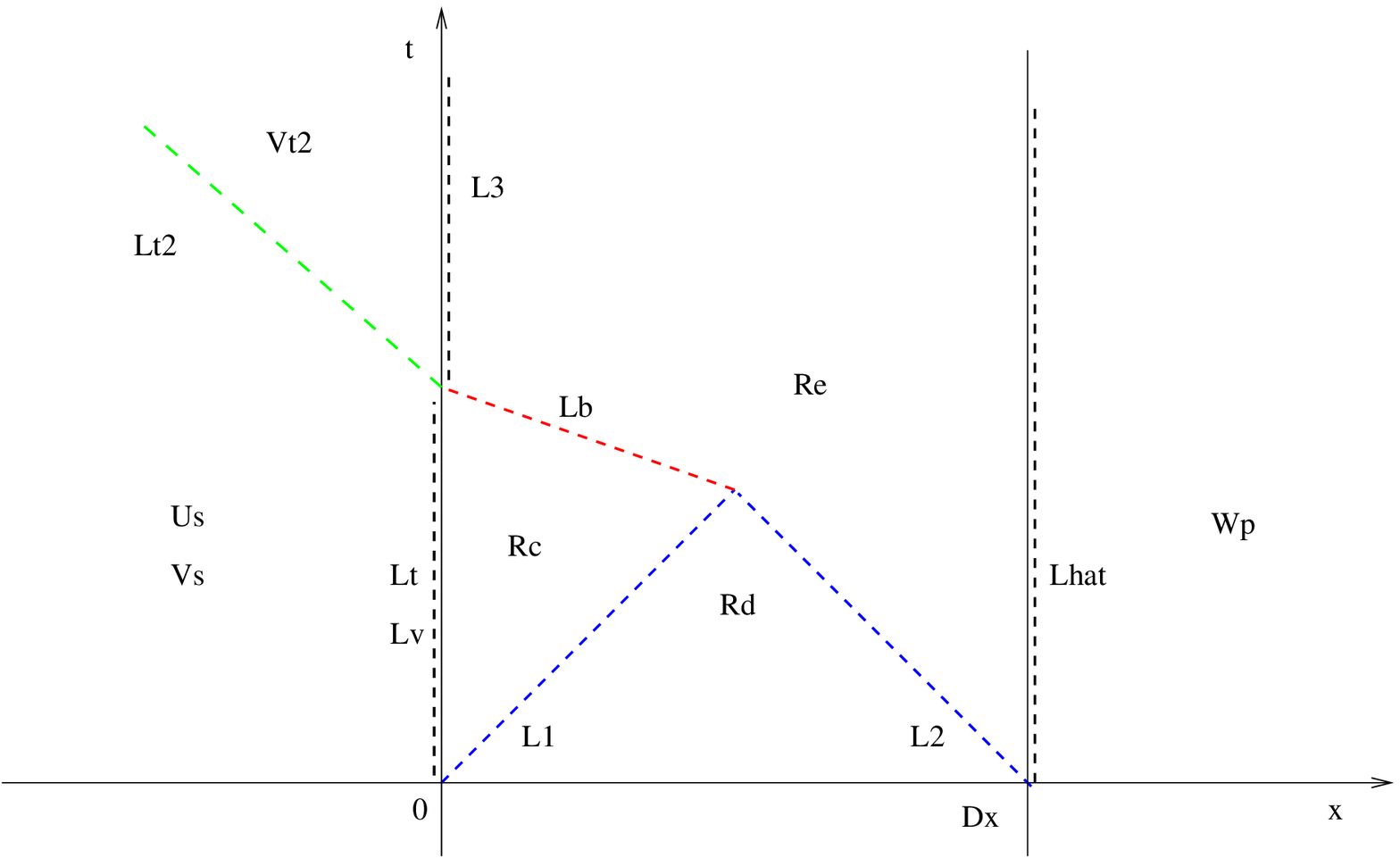}
\end{psfrags}
\end{center}
\caption{CASE (C): Solution of the Riemann problem.}
\label{fig:RiemannPb_C}
\end{figure}

We have 
$$
(\hat z<\sigma<\tilde z) \Rightarrow (f'(\hat z)>f'(\tilde z)) 
\stackrel{\eqref{C_prima_soluz}}{\Rightarrow}
(\hat h'(\hat z)>\tilde h'(\tilde z))
$$ 
and then, referring to Appendix\ \ref{sec:RP_gen_intro}, $L_5$ is a shock. Moreover,
$$
\tilde h(\tilde z)
\stackrel{\eqref{C_prima_soluz}}{=}
f(\wr)
\stackrel{\eqref{(C)Riemann_proof}}{<}
f(\ul)+f(\vl)
\stackrel{\eqref{assumption_uvzw}}{=}
f(\max\{\ul,\vl\})+\min\{f(\ul),f(\vl)\}
\stackrel{\eqref{B_prima_soluz}}{=}
\hat h(\hat z)
$$
and then, by \eqref{RHgen} and $\tilde z>\hat z$, $L_5$ is strictly negative.

\medskip

When $L_5$ reaches the line $\{x=0\}$, three new waves $L_6:=\big((\hat h_2,\hat z_2),(\tilde h,\tilde z)\big)$, $L_7^u:=(\tilde u,\tilde u_2)$, and $L_7^v:=(\tilde v,\tilde v_2)$ start, see Fig.\ \ref{fig:RiemannPb_B}. Then we look for the flux $\hat h_2$ and the constants $\tilde u_2, \tilde v_2, \hat z_2$ such that \eqref{B_def_utilde2}--\eqref{B_def_zhat2} hold true. 
We claim that the solution is
\begin{equation}\label{C_seconda_soluz}
\boxed{
\tilde u_2=\tilde v_2=\hat z_2=\tilde z, \quad
\hat h_2=\tilde h.
}
\end{equation}
We skip the proof that \eqref{B_def_utilde2}--\eqref{B_def_zhat2} are verified and we just prove that $L^u_7, L^v_7$ are strictly negative shocks and that $L_6$ is not created. 

By \eqref{C_seconda_soluz} we get $\tilde u_2>\sigma$ and then $\tilde u=\ul<\sigma<\tilde u_2$. This proves that $L^u_7$ is a shock. Moreover, it is a strictly negative shock since
$$
f(\tilde u)
\stackrel{\eqref{C_prima_soluz}}{=}
f(\ul)
\stackrel{\eqref{(C)Riemann_proof}}{>}
\frac{f(\wr)}{2}
\stackrel{\eqref{C_seconda_soluz}}{=}
f(\tilde u_2).
$$
The result for $L^v_7$ is proved analogously.
$L_6$ instead is not created since by \eqref{C_seconda_soluz} we have $(\hat h_2,\hat z_2)=(\tilde h, \tilde z)$.

The proof is concluded by choosing 
$$
\bar u=\bar v=\tilde u_2=\tilde v_2, \quad \bar z=\tilde z, \quad \bar w=\hat w.
$$
\end{itemize}
\end{proof}
\begin{remark}\label{rem:htotale} (\textit{Modified flux})
The precise expression of $h$ in \eqref{def:h(x,t)} is found \textit{a posteriori}, joining together the solutions
\begin{itemize}
\item \eqref{A_prima_soluz} and \eqref{A_seconda_soluz} for CASE (A), see Fig.\ \ref{fig:RiemannPb_A},  
\item \eqref{B_prima_soluz} and \eqref{B_seconda_soluz}  for CASE (B), see Fig.\ \ref{fig:RiemannPb_B}, 
\item \eqref{C_prima_soluz} and \eqref{C_seconda_soluz} for CASE (C), see Fig.\ \ref{fig:RiemannPb_C}.
\end{itemize}
\end{remark}
\begin{remark} (\textit{Uniqueness})
Once the fluxes $\hat h$, $\hc$, $\tilde h$ are given, the constants $\tilde u$, $\tilde v$, $\hat z$, $\tilde z$, $\hat w$, $\tilde z_2$, $\hat w_2$ are uniquely determined. 
For example, in CASE (A) let us assume by contradiction that $\tilde u\neq \ul$. Then, by \eqref{A_def_utilde} we have that $\tilde u\in N(\ul)$ and then 
\begin{equation}\label{rem:u>sigma}
\tilde u>\sigma.
\end{equation} 
On the other hand, since $\hat h\neq\hc$, \eqref{A_def_zhat} requires $(\hat h,\hat z)\in \mathcal P(\hc,\zc=\sigma)$, which implies $\hat z\leq\sigma$. Finally, by \eqref{A_def_utilde} we get
$$
f(\tilde u)=\min\{G_f(\tilde u,\sigma),G_f(\sigma,\hat z)\}
\stackrel{(\hat z\leq\sigma)}{=}
\min\{f(\sigma),f(\sigma)\}=f(\sigma)
$$
which implies $\tilde u=\sigma$, in contradiction with \eqref{rem:u>sigma}.

Other cases can be managed analogously.
\end{remark}
\section{Riemann and half-Riemann problem}\label{sec:Riemann_and_half-Riemann_intro}
In this appendix we recall some basic notions about Riemann and half-Riemann problem. We also introduce a trivial generalization of the (half) Riemann problem in case of different fluxes on the left and right side.

\subsection{The Riemann problem}\label{sec:RPintro}
Let us consider the following Cauchy problem with Heaviside initial data
\begin{equation}\label{RP_intro}
\left\{
\begin{array}{l}
	\frac{\partial}{\partial t} \rho + \frac{\partial}{\partial x} f(\rho)=0,\\
	\rho(x,0)=\left\{
		\begin{array}{ll} \rhomeno, & \mbox{ if } x< 0,\\
			\rhopiu, & \mbox{ if } x> 0,
		\end{array}\right.
\end{array}\right.
\end{equation}
for some constant initial data $\rhomeno$ and $\rhopiu$ in $[0,\rhomax]$.
The unique weak entropy solution is given by:
\begin{itemize}
\item{If $\rholeft < \rhoright$ (shock wave):}
\begin{equation}
\rho(x,t)=\left\{\begin{array}{ll} \rholeft & \mbox{ for } x< \lambda t\\
				\rhoright & \mbox{ for } x> \lambda t
\end{array}\right.
\quad \textrm{ with } \quad  
\lambda = \Frac{f(\rhoright)-f(\rholeft)}{\rhoright-\rholeft}.
\end{equation}
We have a \textit{shock wave with positive speed} if $f(\rhoright) \geq f(\rholeft)$, and with \textit{negative speed} if $f(\rhoright) \leq f(\rholeft)$.
\item{If $\rholeft>\rhoright$ (rarefaction wave):}
\begin{equation}
\rho(x,t)=\left\{\begin{array}{ll} \rholeft & \mbox{ for } x/t\leq f^\prime(\rholeft),\\
				\psi(x/t) & \mbox{ for } f^\prime(\rholeft) \leq x/t \leq f^\prime(\rhoright),\\
				\rhoright & \mbox{ for } x/t\geq f^\prime(\rhoright),
\end{array}\right.
\end{equation}
where the function $\psi(\xi)$ is defined by the solution of $f^\prime(\psi(\xi))=\xi$. 
We have a {\it rarefaction wave with positive speed} if $f^\prime(\rholeft), f^\prime(\rhoright)\geq 0$, and with
{\it negative speed} if $f^\prime(\rholeft), f^\prime(\rhoright)\leq 0$.
\end{itemize}
\subsection{Half Riemann problem}\label{sec:halfRPintro}
Here we employ the notation $\compl$ defined by \eqref{tau}.
Following \cite{mercier2009JMAA}, we call {\it half Riemann problem} the simple case of an initial-boundary value problem in the quarter of plane $\{x\leq 0\}$ or $\{x\geq 0\}$ when the initial condition is a constant.
The problem is then to find the acceptable boundary condition at $x=0$. 

\medskip

The study of the \textit{left-half} problem is equivalent to searching an artificial right state in the problem \eqref{RP_intro} which lead to waves with negative speed, in order to know which are the states attainable along the line $x=0$.

\begin{definition}\label{def:N}
For any $\rholeft\in[0,\rhomax]$, we define $N(\rholeft)$ to be the set of points $\tilde \rho\in[0,\rhomax]\backslash\{\rholeft\}$ such that the solution to the Riemann problem
\begin{equation}
\left\{
\begin{array}{l}
	\frac{\partial}{\partial t} \rho + \frac{\partial}{\partial x} f(\rho)=0,\\
	\rho(x,0)=\left\{
		\begin{array}{ll} \rholeft, & \mbox{ if } x< 0,\\
			\tilde\rho, & \mbox{ if } x> 0,
		\end{array}\right.
\end{array}\right.
\end{equation}
contains only waves with negative speed.
\end{definition}
It is trivial to verify the following result.
\begin{proposition}\label{prop:N} 
We have
\begin{equation}
N(\rholeft)=
\left\{
\begin{array}{ll}
[\rholeft^\compl,\rhomax]=\{\tilde\rho : \tilde\rho\geq\sigma \mbox{ and } f(\tilde\rho)\leq f(\rholeft)\} & \mbox{ if }\rholeft\leq \sigma, \\ [2mm]
\left[\sigma,\rhomax\right]          & \mbox{ if } \rholeft>\sigma. 
\end{array}
\right.
\end{equation}
\end{proposition}

%
%
%
%
\medskip

The study of the \textit{right-half} problem is equivalent to searching an artificial left state in the problem \eqref{RP_intro} which lead to waves with positive speed, in order to know which are the states attainable along the line $x=0$.

\begin{definition}\label{def:P}
For any $\rhoright\in[0,\rhomax]$, we define $P(\rhoright)$ to be the set of points $\hat \rho\in[0,\rhomax]\backslash\{\rhoright\}$ such that the solution to the Riemann problem
\begin{equation}
\left\{
\begin{array}{l}
	\frac{\partial}{\partial t} \rho + \frac{\partial}{\partial x} f(\rho)=0,\\
	\rho(x,0)=\left\{
		\begin{array}{ll} \hat\rho, & \mbox{ if } x< 0,\\
			\rhoright, & \mbox{ if } x> 0,
		\end{array}\right.
\end{array}\right.
\end{equation}
contains only waves with positive speed.
\end{definition}
Analogously to the previous case, we get the following result.
\begin{proposition}\label{prop:P} 
We have
\begin{equation}
P(\rholeft)=
\left\{
\begin{array}{ll}
[0,\sigma] & \mbox{ if } \rhoright\leq \sigma, \\
\left[0,\rhoright^\compl \right]=\{\hat\rho : \hat\rho < \sigma \mbox{ and } f(\hat\rho)\leq f(\rhoright)\} & \mbox{ if } \rhoright> \sigma. 
\end{array}
\right.
\end{equation}
\end{proposition}


\subsection{The (half) Riemann problem with different fluxes}\label{sec:RP_gen_intro}
To our purposes, we need a trivial generalization of the Riemann problem with different fluxes. Let us consider the following problem
\begin{equation}\label{RPgen}
\left\{
\begin{array}{l}
	\frac{\partial}{\partial t} \rho + \frac{\partial}{\partial x} h(\rho)=0,\\
	\rho(x,0)=\left\{
		\begin{array}{ll} 
		    \rholeft,  & \mbox{ if } x<0,\\
			\rhoright, & \mbox{ if } x>0,
		\end{array}\right.
\end{array}\right.
\mbox{with }
\quad
h(\rho)=
\left\{
\begin{array}{ll}
	\hL(\rho), & \mbox{ if } \rho=\rholeft,\\
	\hR(\rho), & \mbox{ if } \rho=\rhoright,
\end{array}\right.
\end{equation}
where $\hL,\hR:[0,\rhomax]\to\R$ are two $C^1$ fluxes.
This problem is different from what is commonly called ``conservation law with discontinuous flux'', since in our case the discontinuity does not depend  explicitly on $x$. Instead, it depends on the solution $\rho$ itself and it is located at the boundary of two adjacent advection problems of a constant state. Let us clarify this point by computing the entropy solution to \eqref{RPgen}.
\begin{itemize}
\item
If $\hL'(\rho_\ell)>\hR'(\rho_r)$, characteristic lines starting from the line $\{x<0\}$ meet those starting from the line $\{x>0\}$, creating a shock $x=\xi(t)$, see Fig.\ \ref{fig:riemannpb_gen}a. In order to have the mass conserved across the shock, a Rankine-Hugoniot-like condition must be verified. This condition is easy found by standard arguments \cite[Chapt.77]{habermanbook} as 
\begin{equation}\label{RHgen}
\dot \xi(t)=\frac{\hR(\rho_r)-\hL(\rho_l)}{\rho_r-\rho_l}.
\end{equation}
Then we have a \textit{shock wave with positive speed} if $\dot \xi>0$, and with \textit{negative speed} if $\dot\xi<0$.
\item
If instead $\hL'(\rho_\ell)<\hR'(\rho_r)$ a rarefaction is created, see Fig.\ \ref{fig:riemannpb_gen}b. As in the classical Riemann problem, the ``boundaries'' of the rarefaction are given by the lines $x=\hL'(\rholeft)t$ and $x=\hR'(\rhoright)t$. 
Then we have a \textit{rarefaction wave with positive speed} if $\hL'(\rholeft),\hR'(\rhoright)>0$, and with \textit{negative speed} if $\hL'(\rholeft),\hR'(\rhoright)<0$.
Note that the sign of the waves can be computed even if the solution inside the rarefaction is not explicitly given.
\end{itemize}
\begin{figure}[h!]
\begin{center}
\begin{psfrags}
\psfrag{x}{$x$} \psfrag{t}{$t$}
\psfrag{xi}{\ $\xi$}
\psfrag{dl}{\hskip-13pt$\rho=\rholeft$} 
\psfrag{dr}{$\rho=\rhoright$}
\psfrag{rhol}{$\rholeft$} 
\psfrag{rhor}{$\rhoright$}
\psfrag{fr}{$h=\hR$} 
\psfrag{fl}{\hskip-13pt$h=\hL$}
\includegraphics[width=0.9\textwidth]{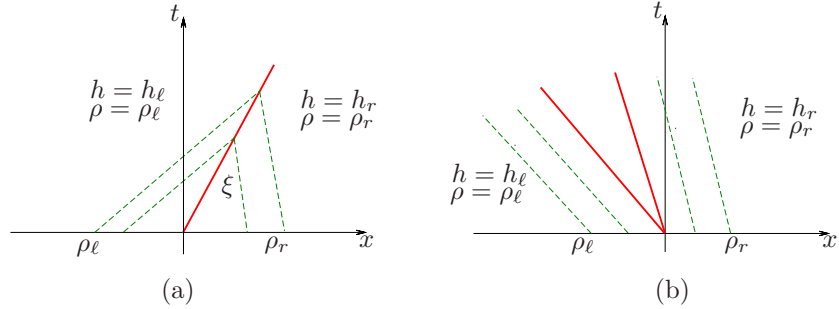}
\end{psfrags}
\end{center}
\hskip0.5cm (a) \hskip6cm (b)
\caption{Riemann problem with different fluxes. (a) Shock, (b) Rarefaction.}
\label{fig:riemannpb_gen}
\end{figure}
Following Definitions \ref{def:N} and \ref{def:P} we introduce the following two definitions:
\begin{definition}\label{def:Ngen}
For any flux $\hL:[0,\rhomax]\to\R$ and density $\rholeft\in[0,\rhomax]$, we define $\mathcal N(\hL,\rholeft)$ to be the set of flux/density couples $(\tilde h,\tilde\rho)$, with $\tilde h:[0,\rhomax]\to\R$ and $\tilde \rho\in[0,\rhomax]\backslash\{\rholeft\}$, such that the solution to the Riemann problem with different fluxes
\begin{equation}\label{RPgenN}
\left\{
\begin{array}{l}
	\frac{\partial}{\partial t} \rho + \frac{\partial}{\partial x} h(\rho)=0,\\
	\rho(x,0)=\left\{
		\begin{array}{ll} 
		    \rholeft,   & \mbox{ if } x<0,\\
			\tilde\rho, & \mbox{ if } x>0,
		\end{array}\right.
\end{array}\right.
\mbox{with }
\quad
h(\rho)=
\left\{
\begin{array}{ll}
	\hL(\rho), & \mbox{ if }      \rho=\rholeft,\\
	\tilde h(\rho), & \mbox{ if } \rho=\tilde\rho,
\end{array}\right.
\end{equation}
contains only waves with negative speed.
\end{definition}
\begin{definition}\label{def:Pgen}
For any flux $\hR:[0,\rhomax]\to\R$ and density $\rhoright\in[0,\rhomax]$, we define $\mathcal P(\hR,\rhoright)$ to be the set of flux/density couples $(\hat h,\hat\rho)$, with $\hat h:[0,\rhomax]\to\R$ and $\hat\rho\in[0,\rhomax]\backslash\{\rhoright\}$, such that the solution to the Riemann problem with different fluxes
\begin{equation}\label{RPgenP}
\left\{
\begin{array}{l}
	\frac{\partial}{\partial t} \rho + \frac{\partial}{\partial x} h(\rho)=0,\\
	\rho(x,0)=\left\{
		\begin{array}{ll} 
		    \hat\rho,  & \mbox{ if } x<0,\\
			\rhoright, & \mbox{ if } x>0,
		\end{array}\right.
\end{array}\right.
\mbox{with }
\quad
h(\rho)=
\left\{
\begin{array}{ll}
	\hat h(\rho), & \mbox{ if }  \rho=\hat\rho,\\
	\hR(\rho),    & \mbox{ if }  \rho=\rhoright,
\end{array}\right.
\end{equation}
contains only waves with positive speed.
\end{definition}

\section*{Acknowledgments}
The authors thank R. Natalini, B. Piccoli and F. S. Priuli for their useful suggestions, and R. M. Colombo for having pointed out the reference \cite{mercier2009JMAA}.


\end{document}